\documentclass[12pt]{article}
\usepackage{epsfig}
\usepackage{amsmath,amsfonts}
\title{Strong Accessibility of Coxeter Groups over Minimal Splittings}
\author{M. Mihalik and S. Tschantz}
\newtheorem{theorem}{Theorem}
\newtheorem{proposition}[theorem]{Proposition}
\newtheorem{lemma}[theorem]{Lemma}
\newtheorem{corollary}[theorem]{Corollary}
\newcounter{remarknum}
\newenvironment{remark}{\addvspace{12pt}\refstepcounter{remarknum}
\noindent{\bf Remark \arabic{remarknum}.}}{\par\addvspace{12pt}}
\newenvironment{proof}{\addvspace{10pt}\noindent{\bf Proof:}}{
$\Box$\par\addvspace{10pt}}
\newcounter{examplenum}
\newenvironment{example}{\addvspace{12pt}\refstepcounter{examplenum}
\noindent{\bf Example \arabic{examplenum}.}}{\par\addvspace{12pt}}
\date{February 25, 2010 
}
\begin{document}
\maketitle

\begin{abstract}
Given a class of groups $\mathcal C$, a group $G$ is strongly accessible over $\mathcal C$ if there is a bound on the number of terms in a sequence $\Lambda_1,\Lambda_2,\ldots , \Lambda_n$ of graph of groups decompositions of $G$ with edge groups in $\mathcal C$ such that 
$\Lambda_1$ is the trivial
decomposition (with 1-vertex) and for $i>1$, $\Lambda_i$ is
obtained from $\Lambda _{i-1}$ by non-trivially and compatibly splitting a
vertex group of $\Lambda_{i-1}$ over a group in $\mathcal C$, replacing this vertex group
by the splitting and then reducing. If $H$ and $K$ are
subgroups of a group $G$ then $H$ is {\it smaller} than $K$ if
$H\cap K$ has finite index in $H$ and infinite index in $K$. The {\it minimal
splitting subgroups of $G$}, are the subgroups $H$
of $G$, such that $G$ splits non-trivially (as an amalgamated
product or HNN-extension) over $H$ and for any other splitting
subgroup $K$ of $W$, $K$ is not smaller than $H$. When $G$ is a finitely generated Coxeter group, minimal splitting subgroups are always finitely generated. Minimal splittings are explicitly or implicitly important aspects of Dunwoody's work on accessibility and the JSJ results of Rips-Sela, Dunwoody-Sageev and Mihalik. Our main results are that Coxeter groups are strongly accessible over minimal splittings and if $\Lambda$ is an irreducible graph of groups decomposition of a Coxeter group with minimal splitting edge groups, then the vertex and edge groups of $\Lambda$ are Coxeter.
\end{abstract}


\section{Introduction}

In \cite{Stallings}, J. Stallings proved that finitely generated
groups with more than one end split non-trivially as an
amalgamated product $A\ast_CB$ (where non-trivial means $A\ne C\ne B$) or an HNN-extension $A\ast_C$ with $C$ a finite group. In about 1970, C. T. C. Wall raised questions about whether or not one could begin with a group $A_0$ and for $i>0$, produce a infinite sequence of non-trivial splittings, $A_i\ast _{C_i}B_i$ or $A_i\ast _{C_i}$  of $A_{i-1}$, with $C_i$ is finite. When such a sequence could not exist, Wall called the group $A_0$, {\it accessible} over such splittings. In
\cite{Dunwoody} M. Dunwoody proved that finitely presented groups
are accessible with respect to splittings over finite groups.
This implies that for a finitely presented group $G$ there is no
infinite sequence $\Lambda_1, \Lambda_2, \ldots $ of graph of
groups decomposition of $G$ such that $\Lambda_1$ is the trivial
decomposition (with 1-vertex) and for $i>1$, $\Lambda_i$ is
obtained from $\Lambda _{i-1}$ by non-trivially splitting a
vertex group over a finite group, replacing this vertex group
by the splitting and then reducing. (For splittings over finite groups there is
never a compatibility problem.) Instead any such sequence of
decompositions must terminate in one in which each vertex group
is either 1-ended or finite and all edge groups are finite. The
class of {\it small} groups is defined in terms of actions on
trees and is contained in the class of groups that contain no
non-abelian free group as a subgroup. In \cite{BestvinaFeighn}, M.
Bestvina and M. Feighn show that for a finitely presented group
$G$ there is a bound $N(G)$ on the number of edges in a
reduced graph of groups decomposition of $G$, when edge groups
are small. Limits of this sort are generally called
``accessibility" results. If $\mathcal C$ is a class of groups then call a graph of groups decomposition of a group $G$ with edge groups in $\mathcal C$ a $\mathcal C$-{\it decomposition} of $G$. A group $G$ is called {\it strongly
accessible} over $\cal C$ if there is a bound on the number of terms in a sequence $\Lambda _1, \Lambda _2, \ldots , \Lambda_n$ of
$\mathcal C$-decompositions of $G$, such that $\Lambda_1$ 
is the trivial decomposition, and for $i> 1$, $\Lambda
_i$ is obtained from $\Lambda _{i-1}$ by replacing a vertex group of $\Lambda _{i-1}$ with a compatible splitting $A\ast_CB$ or $A\ast_C$ ($C\in \cal C$) and then reducing. We call a group $G$ 
{\it accessible} over a class of groups $\cal C$ if there is a bound $N(G)$ on the number of edge groups in a reduced graph of groups decomposition of $G$ with edge groups in $\cal C$. Certainly strong accessibility implies accessibility. Dunwoody's theorem is a strong accessibility result for finitely presented groups over the class of finite groups. We know of no example where accessibility and strong accessibility are different. 

In this paper, we produce accessibility results for
finitely generated Coxeter groups. In analogy with the 1-ended
assumptions of Rips-Sela \cite{RipsSela}, and the minimality
assumptions of \cite{DunwoodySageev}, we consider the class $M(W)$
of minimal splitting subgroups of $W$. If $H$ and $K$ are
subgroups of a group $W$ then $H$ is {\it smaller} than $K$ if
$H\cap K$ has finite index in $H$ and infinite index in $K$. If
$W$ is a group, then define $M(W)$, the set of {\it minimal
splitting subgroups of $W$}, to be the set of all subgroups $H$
of $W$, such that $W$ splits non-trivially (as an amalgamated
product or HNN-extension) over $H$ and for any other splitting
subgroup $K$ of $W$, $K$ is not smaller than $H$. 

\begin{remark} 
A minimal splitting subgroup of a finitely generated Coxeter group $W$ is finitely generated. This follows from remark 1 of \cite{MTVisual}. Suppose  $A\ast_CB$ is a non-trivial splitting of $W$ and $C$ is not finitely generated. There a reduced visual decomposition of $W$ with (visual and hence finitely generated) edge group $E$ such that a conjugate of $E$ is a subgroup of $C$. Hence some conjugate of $E$ is smaller than $C$. 
\end{remark}

Finite
splitting subgroups are always minimal and if a group is 1-ended,
then any 2-ended splitting subgroup is minimal. Our main theorem
is:

\begin{theorem}\label{Main} 
Finitely generated Coxeter groups are strongly accessible over
minimal splittings.
\end{theorem}

Our basic reference for Coxeter groups is Bourbaki
\cite{Bourbaki}. A {\it Coxeter presentation} is given by
$$\langle S: m(s,t)\ (s,t\in S,\ m(s,t)<\infty )\rangle$$ where
$m:S^2\to \{1,2,\ldots ,\infty \}$ is such that $m(s,t)=1$ iff
$s=t$ and $m(s,t)=m(t,s)$. The pair $(W,S)$ is called a {\it
Coxeter system}. In the group with this presentation, the elements
of $S$ are distinct elements of order 2 and a product $st$ of
generators has order $m(s,t)$. Distinct generators commute if and
only if $m(s,t)=2$. A subgroup of $W$ generated by a subset $S'$
of $S$ is called {\it special} or {\it visual}, and the pair
$(\langle S'\rangle, S')$ is a Coxeter system with $m':(S')^2\to \{1,2,\ldots ,\infty\}$ the
restriction of $m$. A simple analysis of a Coxeter
presentation allows one to construct all decompositions of $W$
with only visual vertex and edge groups from that Coxeter
presentation. In \cite{MTVisual}, the authors show that for any finitely generated
Coxeter system $(W,S)$ and any graph of groups decomposition $\Lambda$ of $W$, there is an associated ``visual" graph
of groups decomposition $\Psi$ of $W$ with edge and vertex groups
visual, and such that each vertex (respectively edge) group of $\Psi$ is contained in a conjugate of a vertex
(respectively edge) group of $\Lambda$. This result
is called ``the visual decomposition theorem for finitely
generated Coxeter groups", and we say $\Psi$ is {\it a visual decomposition for $\Lambda$}. Clearly accessibility of finitely generated Coxeter groups is not violated by only visual decompositions. But, we give an example in \cite{MTVisual}, of a
finitely generated Coxeter system $(W,S)$ and a sequence
$\Lambda_i$ ($i\geq 1$) of (non-visual) reduced graph of groups
decompositions of $W$, such that $\Lambda _i$ has $i$-edge groups
and, for $i>1$, $\Lambda _i$ is obtained by compatibly splitting
a vertex group of $\Lambda _{i-1}$. Hence, even in the light of
the visual decomposition theorem and our accessibility
results here, there is no accessibility for Coxeter groups
over arbitrary splittings.

Theorem \ref{Main} implies there are irreducible decompositions of finitely generated Coxeter groups, with minimal splitting edge groups. Our next result implies that any such irreducible decomposition
has an ``equivalent" visual counterpart.

\begin{theorem}\label{Close}
Suppose $(W,S)$ is a Coxeter system and $\Lambda$ is a reduced
graph of groups decomposition of $W$ with $M(W)$ edge groups. If
$\Lambda$ is irreducible with respect to $M(W)$ splittings, and
$\Psi$ is a reduced graph of groups decomposition such that each edge group of $\Psi$ is in $M(W)$, each vertex group of $\Psi$ is a subgroup of a conjugate of a vertex group of  
$\Lambda$, and each edge group of $\Lambda$ contains a conjugate of an edge group of $\Psi$ (in particular if $\Psi$ is a reduced visual graph of groups decomposition for $(W,S)$ derived from $\Lambda$ as in the main theorem of \cite{MTVisual}), then
\begin{enumerate}
\item  $\Psi$ is irreducible with respect to $M(W)$ splittings

\item  There is a (unique) bijection $\alpha$ of the vertices
of $\Lambda$ to the vertices of $\Psi$ such that for each vertex
$V$ of $\Lambda$, $\Lambda(V)$ is conjugate to $\Psi(\alpha (V))$

\item  When $\Psi$ is visual, each edge group of $\Lambda$ is conjugate to a visual
subgroup for $(W,S)$.
\end{enumerate}
\end{theorem}

The vertex groups of $\Lambda$ in theorem \ref{Close} are Coxeter, and when $W$ is not indecomposable, they have fewer generators than there are in $S$. Hence they have irreducible decompositions of the same type. As the number of Coxeter generators decreases each time we pass from a non-indecomposable vertex group to a vertex group of an irreducible decomposition with minimal splitting edge groups for that vertex group, eventually this must process must terminate with (up to conjugation) irreducible visual subgroups of $(W,S)$. These terminal groups are maximal FA subgroups of $W$ and must be conjugate to the visual subgroups of $W$ determined by maximal complete subsets of the presentation diagram $\Gamma (W,S)$ (see \cite{MTVisual}).

The paper is laid out as follows: in \S 2 we state the visual
decomposition theorem and review the basics of graphs of groups
decompositions.

In \S 3, we list several well-known technical facts about Coxeter
groups. \S 3 concludes with an argument that shows an infinite
subgroup of a finitely generated Coxeter group $W$ (with Coxeter
system $(W,S)$), containing a visual finite index subgroup
$\langle A\rangle$ ($A\subset S$) decomposes as $\langle
A_0\rangle \times F$ where $A_0\subset A$ and $F$ is a finite
subgroup of a finite group $\langle D\rangle$ where $D\subset S$
and $D$ commutes with $A_0$. This result makes it possible for us
to understand arbitrary minimal splitting subgroups of $W$ in our
analysis of strong accessibility.

In \S 4, we begin our analysis of $M(W)$ by classify the visual
members of $M(W)$ for any Coxeter system $(W,S)$. 
Proposition \ref{L24N} shows that for a non-trivial splitting $A\ast_CB$ of a finitely generated Coxeter group $W$ over a non-minimal group $C$, there is a splitting of $W$ over a minimal splitting subgroup $M$, such that $M$ is smaller than $C$. I.e. all non-trivial splittings of a finitely generated Coxeter group are ``refined" by minimal splittings. Theorem \ref {T1N} is the analogue of theorem \ref{artificial} (from \cite{MTVisual}), when edge groups of a graph of groups decomposition of a finitely generated Coxeter group are minimal splitting subgroups.  The implications with this additional ``minimal splitting" hypothesis far exceed the conclusions of theorem \ref{artificial}  and supply one of the more important technical results of paper. Roughly speaking, proposition \ref{P3N} says that any graph of groups decomposition of a finitely generated Coxeter group with edge groups equal to minimal splitting subgroups of the Coxeter group is, up to ``artificial considerations", visual. Proposition \ref{P3N} gives  another key idea towards the proof of the main theorem. It allows us to define a descending sequence of positive integers corresponding to a given sequence of graphs of groups as in the main theorem. 
Finally, theorem \ref{vismin} is a
minimal splitting version of the visual decomposition theorem of \cite{MTVisual}.

In \S 5, we define what it means for a visual decomposition of a
Coxeter group $W$, with $M(W)$ edge groups, to look irreducible
with respect to $M(W)$ subgroups. We show that a visual
decomposition looks irreducible if and only if it is irreducible.
This implies that all irreducible visual decompositions of a
Coxeter group can be constructed by an elementary algorithm.
Our main results, theorems \ref{Main} and \ref{Close} are proved
in \S5.

In the final section, \S 6, we  begin with a list  of generalizations of our
results that follow from the techniques of the paper. Then, we give an analysis of minimal splitting subgroups of ascending HNN extensions, followed by a complete analysis of minimal splittings of general finitely generated groups that contain no  non-abelian free group. This includes an analysis of Thompson's group $F$. We conclude
with a list of questions.

\section{Graph of Groups and Visual Decompositions}

Section 2 of \cite{MTVisual} is an introduction to graphs of
groups that is completely sufficient for our needs in this paper.
We include the necessary terminology here. A graph of groups
$\Lambda$ consists of a set $V(\Lambda)$ of vertices, a set
$E(\Lambda)$ of edges, and maps $\iota,\tau:E(\Lambda)\to
V(\Lambda)$ giving the initial and terminal vertices of each edge
in a connected graph, together with vertex groups $\Lambda(V)$
for $V\in V(\Lambda)$, edge groups $\Lambda(E)$ for $E\in
E(\Lambda)$, with $\Lambda(E)\subset\Lambda(\iota(E))$ and an injective
group homomorphism $t_E:\Lambda(E)\to\Lambda(\tau(E))$, called
the edge map of $E$ and denoted by $t_E:g\mapsto g^{t_E}$. The
fundamental group $\pi(\Lambda)$ of a graph of groups $\Lambda$
is the group with presentation having generators the disjoint
union of $\Lambda(V)$ for $V\in V(\Lambda)$, together with a
symbol $t_E$ for each edge $E\in E(\Lambda)$, and having as
defining relations the relations for each $\Lambda(V)$, the
relations $gt_E=t_Eg^{t_E}$ for $E\in E(\Lambda)$ and
$g\in\Lambda(\iota(E))$, and relations $t_E=1$ for $E$ in a given
spanning tree of $\Lambda$ (the result, up to isomorphism, is
independent of the spanning tree taken).

If $V$ is a vertex of a graph of groups decomposition $\Lambda$ of
a group $G$ and $\Phi$ is a decomposition of $\Lambda(V)$ so that
for each edge $E$ of $\Lambda$ adjacent to $V$, $\Lambda(E)$ is
$\Lambda (V)$-conjugate to a subgroup of a vertex group of
$\Phi$, then $\Phi$ is {\it compatible} with $\Lambda$. Then $V$
can be replaced by $\Phi$ to form a finer graph of groups
decomposition of $G$.

A graph of groups is {\it reduced} if no edge between distinct
vertices has edge group the same as an endpoint vertex group. If
a graph of groups is not reduced, then we may collapse a vertex
across an edge, giving a smaller graph of groups decomposition of
the group.

If there is no non-trivial homomorphism of a group to the infinite
cyclic group $\mathbb Z$, then a graph of groups decomposition of
the group cannot contain a loop. In this case, the graph is a
tree. In particular, any graph of groups decomposition of a
Coxeter group has underlying graph a tree.

Suppose $\langle S: m(s,t)\ (s,t\in S,\ m(s,t)<\infty )\rangle$
is a Coxeter presentation for the Coxeter group $W$. The {\it
presentation diagram} $\Gamma(W,S)$ of $W$ with respect to $S$
has vertex set $S$ and an undirected edge labeled $m(s,t)$
connecting vertices $s$ and $t$ if $m(s,t)<\infty$. It is evident
from the above presentation that if a subset $C$ of $S$ separates
$\Gamma (W,S)$, $A$ is $C$ union some of the components of
$\Gamma -C$ and $B$ is $C$ union the rest of the components, then
$W$ decomposes as $\langle A\rangle \ast _{\langle C\rangle } \langle B\rangle$. This generalizes to graphs of
groups decompositions of Coxeter groups where each vertex and
edge group is generated by a subset of $S$. We say that $\Psi$ is a {\it visual graph of groups decomposition of} $W$ (for a given $S$), if each vertex and edge group of $\Psi$ is a special subgroup of $W$, the injections of each edge group into its endpoint vertex groups are given simply by inclusion, and the fundamental group of  $\Psi$ is isomorphic to $W$ by the homomorphism induced by the inclusion map of vertex groups into $W$. If $C$ and $D$ are subsets of $S$, then we say $C$ {\it separates} $D$ {\it in} $\Gamma$ if there are points $d_1$ and $d_1$ of $D-C$, such that any path in $\Gamma$ connecting $d_1$ and $d_2$ contains a point of $C$.

The following lemma of \cite{MTVisual} makes it possible to understand when a graph of groups with special subgroups has fundamental group $W$.

\begin{lemma}\label{MT2} 
Suppose $(W,S)$ is a Coxeter system.  A graph of groups $\Psi$
with graph a tree, where each vertex group and edge group is a
special subgroup and each edge map is given by inclusion, is a
visual graph of groups decomposition of $W$ iff each edge in the
presentation diagram of $W$ is an edge in the presentation diagram
of a vertex group and, for each generator $s\in S$, the set of
vertices and edges with groups containing $s$ is a nonempty
subtree in $\Psi$.
\end{lemma}

In section 4 we describe when visual graph of groups decompositions with minimal splitting edge groups are irreducible with respect to splittings over minimal splitting subgroups. The next lemma follows easily from lemma \ref{MT2} and helps make that description possible.

\begin{lemma} \label{MT3} 
Suppose $\Psi$ is a visual graph of groups decomposition for the finitely generated Coxeter system $(W,S)$, $V\subset S$ is such that $\langle V\rangle$ is a vertex group of $\Psi$ and $E\subset V$ separates $V$ in $\Gamma (W,S)$. Then $\langle V\rangle$ splits over $\langle E\rangle$, non-trivially and compatibly with $\Psi$ to give a finer visual decomposition for $(W,S)$ if and only if there are subsets $A$ and $B$ of $S$ such that $A$ is equal to $E$ union (the vertices of)  some of the components of $\Gamma -E$, $B$ is $E$ union the rest of the components of $\Gamma-E$, $A\cap V\ne E\ne B\cap V$, and for each edge $D$ of $\Psi$ which is adjacent to $V$, and $D_S\subset S$ such that $\langle D_S\rangle =\Psi(D)$, we have $D_S\subset A$ or $D_S\subset B$. The $\Psi$-compatible splitting of $\langle V\rangle$ is $\langle A\cap V\rangle\ast _{\langle E\rangle } \langle B\cap V\rangle$.
\end{lemma}

The main theorem of \cite{MTVisual} is ``the visual decomposition theorem for finitely generated Coxeter groups":

\begin{theorem}\label{MT1} 
Suppose $(W,S)$ is a Coxeter system and $\Lambda$ is a graph of
groups decomposition of $W$. Then $W$ has a visual graph of
groups decomposition $\Psi$, where each vertex (edge) group of $\Psi$ is
a subgroup of a conjugate of a vertex (respectively edge) group of $\Lambda$.
Moreover, $\Psi$ can be taken so that each special subgroup of
$W$ that is a subgroup of a conjugate of a vertex group of
$\Lambda$ is a subgroup of a vertex group of $\Psi$.
\end{theorem}

If $(W,S)$ is a finitely generated Coxeter system, $\Lambda$ is a graph of groups decomposition of $W$ and $\Psi$ satisfies the conclusion of theorem \ref{MT1} (including the moreover clause)  and then $\Psi$ is called {\it a visual decomposition from} $\Lambda$ (see \cite{MTVisual}). In remark 1 of \cite{MTVisual}, it is shown that if $\Lambda$ is reduced  and $\Psi$ is a visual decomposition from $\Lambda$ then for  any edge $E$ of $\Lambda$ there is an edge $D$ of $\Psi$ such that $\Psi (D)$ is conjugate to a subgroup of $\Lambda (E)$.

If a group $G$ decomposes as $A\ast _CB$ and $H$ is a subgroup of
$B$, then the group $\langle A\cup H\rangle$ decomposes as $A\ast
_C\langle C\cup H\rangle$. Furthermore, $G$ decomposes as $\langle
A\cup H\rangle _{\langle C\cup H\rangle}B$, giving a somewhat
``artificial" decomposition of $G$. In \cite{MTVisual}, this  idea is used on a
certain Coxeter system $(W,S)$ to produce reduced graph of groups
decompositions of $W$ with arbitrarily large numbers of edges.

The following theorem of \cite{MTVisual} establishes limits on how
far an arbitrary graph of groups decomposition for a finitely
generated Coxeter system can stray from a visual decomposition
for that system.

\begin{theorem}\label{artificial} 
Suppose $(W,S)$ is a finitely generated Coxeter system, $\Lambda$
is a graph of groups decomposition of $W$ and $\Psi$ is a reduced
graph of groups decomposition of $W$ such that each vertex group of $\Psi$ is a subgroup of a conjugate of a vertex group of $\Lambda$. Then for each vertex
$V$ of $\Lambda$, the vertex group $\Lambda (V)$, has a graph of
groups decomposition $\Phi_V$ such that each vertex group of
$\Phi _V$ is either

(1) conjugate to a vertex group of $\Psi$ or

(2) a subgroup of $v\Lambda (E)v^{-1}$ for some $v\in \Lambda
(V)$ and $E$ some edge of $\Lambda$ adjacent to $V$.
\end{theorem}

When $\Psi$ is visual, vertex groups of the first type in theorem \ref{artificial} are
visual and those of the second type seem somewhat artificial. In section 4 we prove theorem \ref{T1N} which shows that if the edge groups of the decomposition $\Lambda$ in theorem \ref{artificial} are minimal splitting subgroups of $W$, then the decompositions $\Phi_V$ are compatible with $\Lambda$ and part (2) of the conclusion can be significantly enhanced.

\begin{lemma}\label{Sub}
If $\Lambda$ is a reduced graph of groups decomposition of a
group $G$, $V$ and $U$ are vertices of $\Lambda$ and $g\Lambda
(V)g^{-1}\subset\Lambda (U)$ for some $g\in G$, then $V=U$. If
additionally $\Lambda$ is a tree, then $g\in \Lambda(V)$.
$\square$
\end{lemma}

If $W$ is a finitely generated Coxeter group then since $W$ has a
set of order 2 generators, there is no non-trivial homomorphism
from $W$ to $\mathbb Z$. Hence any graph of groups decomposition
of $W$ is a tree. If $C\in M(W)$ and $W$ is finitely generated,
then theorem \ref{MT1} implies that $C$ contains a subgroup of
finite index which is isomorphic to a Coxeter group and so there
is no non-trivial homomorphism of $C$ to $\mathbb Z$.

The following is an easy exercise in the theory of graph of
groups or more practically it is a direct consequence of the
exactness of the Mayer-Viatoris sequence for a pair of groups.

\begin{lemma}\label{Z} 
Suppose the group $W$ decomposes as $A\ast _CB$ and there is no
non-trivial homomorphism of $W$ or $C$ to $\mathbb Z$. Then there
is no non-trivial homomorphism of $A$ or $B$ to $\mathbb Z$.
$\square$
\end{lemma}

\begin{corollary}\label{CZ} 
Suppose $W$ is a finitely generated Coxeter group and $\Lambda$
is a graph of groups decomposition of $W$ with each edge group in
$M(W)$, then any graph of groups decomposition of a vertex group
of $\Lambda$ is a tree. $\square$
\end{corollary}

\section{Preliminary results}

We list some results used in this paper. Most can be found in
\cite{Bourbaki}.

\begin{lemma} \label{Order} 
Suppose $(W,S)$ is a Coxeter system and $P=\langle
S:(st)^{m(s,t)}$ for $m(s,t)<\infty\rangle$ (where $m:S^2\to
\{1,2,\ldots ,\infty\}$ ) is a Coxeter presentation for $W$. If
$A$ is a subset of $S$, then $(\langle A\rangle, A)$ is a Coxeter
system with Coxeter presentation $\langle A:(st)^{m'(s,t)}$ for
$m'(s,t) <\infty \rangle$ (where $m'=m\vert _{A^2}$). In
particular, if $\{s,t\}\subset S$, then the order of $(st)$ is
$m(s,t)$. $\square$
\end{lemma}

The following result is due to Tits:

\begin{lemma} \label{Tits} 
Suppose $(W,S)$ is a Coxeter system and $F$ is a finite subgroup
of $W$ then there is $A\subset S$ such that $\langle A\rangle$ is
finite and some conjugate of $F$ is a subgroup of $\langle
A\rangle$. $\square$
\end{lemma}

If $A$ is a set of generators for a group $G$, the {\it Cayley
graph} ${\cal K}(G,A)$ of $G$ with respect to $A$ has $G$ as
vertex set and a directed edge labeled $a$ from $g\in G$ to $ga$
for each $a\in A$. The group $G$ acts on the left of $\cal K$.
Given a vertex $g$ in $\cal K$, the edge paths in $\cal K$ at $g$
are in 1-1 correspondence with the words in the letters $A^{\pm
1}$ where the letter $a^{-1}$ is used if an edge labeled $a$ is
traversed opposite its orientation. Note that for a Coxeter
system $(W,S)$, and $s\in S$, $s=s^{-1}$. It is standard to
identify the edges labeled $s$ at $x$ and $s$ at $xs$ in ${\cal
K}(W,S)$ for each vertex $x$, of $\cal K$ and each $s\in S$ and to
ignore the orientation on the edges. Given a group $G$ with
generators $A$, an $A$-{\it geodesic} for $g\in G$ is a shortest
word in the letters $A^{\pm 1}$ whose product is $g$. A geodesic
for $G$ defines a geodesic in $\cal K$ for each vertex $g\in G$.
Cayley graphs provide and excellent geometric setting for many of
the results in this section.

The next result is called the deletion condition for Coxeter
groups. An elementary proof of this fact, based on Dehn diagrams,
can be found in \cite{MTVisual}.

\begin{lemma}\label{Del}
{\bf The Deletion Condition} Suppose $(W,S)$ is a Coxeter system
and $a_1\cdots a_n$ is a word in $S$ which is not geodesic. Then
for some $i<j$, $a_i\cdots a_j=a_{i+1}\cdots a_{j-1}$. I.e. the
letters $a_i$ and $a_j$ can be deleted. $\square$
\end{lemma}

The next collection of lemmas can be derived from the deletion
condition.

\begin{lemma} \label{Double}
Suppose $(W,S)$ is a Coxeter system and $A$ and $B$ are subsets
of $S$. Then for any $w\in W$ there is a unique shortest element,
$d$, of the double coset $\langle A\rangle w\langle B\rangle$. If
$\delta$ is a geodesic for $d$, $\alpha$ is an $A$-geodesic, and
$\beta$ is a $B$-geodesic, then $(\alpha, \delta)$ and $(\delta,
\beta)$ are geodesic. $\square$
\end{lemma}

\begin{lemma} \label{Kil} 
Suppose $(W,S)$ is a Coxeter system, $w\in W$, $I$ and $J\subset
S$, and $d$ is the minimal length double coset representative in
$\langle I\rangle w\langle J\rangle$. Then $\langle I\rangle\cap
d\langle J\rangle d^{-1}=\langle K\rangle$ for $K=I\cap
(dJd^{-1})$ and, $d^{-1}\langle K\rangle d=\langle J\rangle \cap
(d^{-1}\langle I\rangle d)=\langle K'\rangle$ for $K'=J\cap
d^{-1}Id=d^{-1}Kd$. In particular, if $w=idj$ for $i\in \langle
I\rangle$ and $j\in \langle J\rangle$ then $\langle I\rangle \cap
w\langle J \rangle w^{-1}=i\langle K\rangle i^{-1}$ and $\langle
J\rangle \cap w^{-1}\langle I\rangle w=j^{-1}\langle K'\rangle
j$. $\square$
\end{lemma}

\begin{lemma}\label{Back}
Suppose $(W,S)$ is a Coxeter system, $A$ is a subset of $S$ and
$\alpha$ is an $S$-geodesic. If for each letter $a\in A$, the word
$(\alpha ,a)$ is not geodesic, then the group $\langle A\rangle$
is finite. $\square$
\end{lemma}

\begin{lemma}\label{Extend} 
Suppose $(W,S)$ is a Coxeter system and $x\in S$. If $\alpha$ is a
geodesic in $S-\{x\}$, then the word $(\alpha, x)$ is geodesic.
$\square$
\end{lemma}

If $(W,S)$ is a Coxeter system and $w\in W$ then the deletion
condition implies that the letters of $S$ used to compose an
$S$-geodesic for $w$ is independent of which geodesic one
composes for $w$. We define $lett(w)_S$ to be the subset of $S$
used to composes a geodesic for $w$, or when the system is
evident we simply write $lett(w)$.

\begin{lemma}\label{L56}
Suppose $(W,S)$ is a Coxeter system, $w\in W$, $b\in S-lett(w)$,
and $bwb\in \langle lett(w)\rangle$ then $b$ commutes with $lett
(w)$.$\square$
\end{lemma}


The next lemma is technical but critical to the main
results of the section.

\begin{lemma} \label{L54} 
Suppose $(W,S)$ is a finitely generated Coxeter system and
$A\subset S$ such that $\langle A\rangle$ is infinite and there
is no non-trivial $F\subset A$ such that $\langle F\rangle$ is finite and
$A-F$ commutes with $F$. Then there is an infinite $A$-geodesic
$\alpha$, such that each letter of $A$ appears infinitely many
times in $\alpha$.
\end{lemma}

\begin{proof} The case when $\langle
A\rangle$ does not (visually) decompose as $\langle A-U\rangle\times \langle
U \rangle$ for any non-trivial $U\subset A$, follows from lemma 1.15 of \cite{MRT}. The general case follows since once the
irreducible case is established, one can interleave geodesics from
each (infinite) factor of a maximal visual direct product
decomposition of $\langle A\rangle$. I.e. if $\langle A\rangle
=\langle A-U\rangle\times \langle U\rangle$, $(x_1,x_2,\ldots )$
and $(y_1,y_2,\ldots)$ are $U$ and $A-U$-geodesics respectively,
then the deletion condition implies $(x_1,y_1,x_2,y_2,\ldots)$ is
an $A$-geodesic.
\end{proof}
 
 \begin{remark}\label{Split} 
Observe that if $(W,S)$ is a Coxeter system, and $W=\langle
F\rangle \times \langle G\rangle=\langle H\rangle \times \langle
I\rangle$ for $F\cup G=S=H\cup I$. Then $W=\langle F\cup H\rangle
\times \langle G\cap I\rangle$ and $\langle F\cup H\rangle
=\langle F\rangle \times \langle H-F\rangle$. In particular, for
$A\subset S$, there is a unique largest subset $C\subset A$ such
that $\langle A\rangle =\langle A-C\rangle \times \langle
C\rangle$ and $\langle C\rangle$ is finite. Define $T_{(W,S)}(A)\equiv
C$ and $E_{(W,S)}(A)\equiv A-C$. When the system is evident we simply write $T_W(A)$ and $E_W(A)$.
\end{remark}

 For a Coxeter system $(W,S)$ and $A\subset S$, let $lk_2(A, (W,S))$
({\it the 2-link of $A$} in the system $(W,S)$) be the set of all $s\in S-A$ that commute
with $A$. For consistency we define $lk_2(\emptyset ,(W,S))=S$. When the system is evident we simply write $lk_2(A)$. 
In the presentation diagram $\Gamma (W,S)$, $lk_2(A)$
is the set of all vertices $s\in S$ such that $s$ is connected to
each element of $A$ by an edge labeled 2.
 
If $G$ is a group with generating set $S$ and $u$ is an $S$-word, denote by $\bar u$  the element of $G$ represented by $u$. 
 
\begin{lemma} \label{Fin2} 
Suppose $(W,S)$ is a Coxeter system, $A\subset S$, and $r$ is an
$A$-geodesic such that each letter of $A$ appears infinitely
often in $r$. If $r$ can be partitioned as $(r_1,r_2,\ldots )$
and $w\in W$ is such that $w\bar r_iw^{-1}=s_i$, $\vert s _i\vert = \vert \bar r_i\vert$, and
$(\beta,r_i,r_{i+1},\ldots )$  and $(r_1,\ldots ,r_i,\beta ^{-1})$
are geodesic for all $i$ where $\beta$ is a geodesic for $w$,
then $w\in \langle A\cup lk_2(A)\rangle$.
\end{lemma}

\begin{proof} If $w$ is a minimum length counter-example,
then by lemma \ref{L56}, $\vert w\vert>1$. Say $(w_1,\ldots ,w_n)$
is a geodesic for $w$. For all $m$, $(w_1,\ldots ,w_n,r_1,\ldots ,
r_m,w_n)$ is not geodesic and the last $w_n$ deletes with one of
the initial $w_i$. For some $i\in \{1,\ldots ,n\}$, there are
infinitely many $m$ such that the last $w_n$ deletes with $w_i$.
Say this set of such $m$ is $\{m_1,m_2,\ldots \}$ (in ascending order). Then $w_n$
commutes with $\bar r_{m_j+1} \bar r_{m_j+2} \cdots \bar
r_{m_{j+1}}$ for all $j$. By lemma \ref{L56}, $w_n\in A_0\cup
lk_2(A_0)$. Then $w'=w_1\cdots w_{n-1}$ is shorter than $w$ and
satisfies the hypothesis of the lemma with $r$ replaced by
$r'=(r_1',r_2',\ldots )$ where $r_i'=(r_{m_i+1}, r_{m_i+2},\ldots
,r_{m_{i+1}})$. By the minimality of $w$ , $w'\in \langle A_0\cup
lk_2(A_0)\rangle$, and so $w\in \langle A\cup
lk_2(A)\rangle$. 
\end{proof}

The next result is analogous to classical results (see V. Deodhar \cite{Deodhar}).

\begin{lemma}\label{LFin} 
Suppose $(W,S)$ is a finitely generated Coxeter system, $A$ and $B$ are subsets of
$S$, $u$ is a shortest element of
the double coset $\langle B\rangle g\langle A\rangle$, and
$g\langle A\rangle g^{-1}\subset \langle B\rangle$. Then $uAu^{-1}\subset B$ and $lett(u)\subset lk_2(E_W(A))$. In particular, $uxu^{-1}=x$ for all $x\in E_W(A)$ and $E_W(A)\subset E_W(B)$. If additionally, $g\langle A\rangle g^{-1}=\langle B\rangle$, then $uAu^{-1}=B$ and $E_W(A)=E_W(B)$.
\end{lemma}

\begin{proof}  Note that $g\langle A\rangle g^{-1}=bua\langle A\rangle a^{-1}u^{-1}b^{-1}\subset \langle B\rangle$ for some $a\in \langle A\rangle$ and $b\in \langle B\rangle$. Then $u\langle A\rangle u^{-1}\subset \langle B\rangle$. 
By lemma \ref{Kil}, $u\langle A\rangle u^{-1}=u\langle A\rangle u^{-1}\cap \langle B\rangle=\langle (uAu^{-1})\cap B\rangle$ and so $\langle A\rangle =\langle  A\cap u^{-1}Bu\rangle$ and $A\subset u^{-1}Bu$ so that $uAu^{-1}\subset B$. 

If $E(A)=\emptyset$ there is nothing more to prove. Otherwise,  lemma \ref{L56} implies there is a geodesic $\alpha$ in the letters of $E_W(A)$, such that each letter of $E_W(A)$ appears infinitely often in $\alpha$. By lemma \ref{Fin2} (with partitioning $r_i$ of length 1), $lett(u)\subset E_W(A)\cup lk_2(E_W(A))$. By the definition of $u$, no geodesic for $u$ can end in a letter of $A$ and so $lett(u)\subset lk_2(E_W(A))$. Then $E_W(A)\subset B$ so $E_W(A)\subset E_W(B)$. 

Now assume $g\langle A\rangle g^{-1}=\langle B\rangle$.  Then as $u^{-1}$ is the shortest element of the double coset $\langle A\rangle g^{-1}\langle B\rangle$, we have $u^{-1}Bu\subset A$ so $uAu^{-1}=B$, and we have $E_W(B)\subset E_W(A)$ so $E_W(A)=E_W(B)$.
\end{proof}

\begin{proposition} \label{Index} 
Suppose $(W,S)$ is a Coxeter system, $B$ is an infinite subgroup
of $W$ and $A\subset S$ such that $\langle A\rangle$ has finite
index in $B$. Then $B=\langle A_0\rangle \times C$ for $A_0\subset
A$ and $C$ a finite subgroup of $\langle lk_2(A_0)\rangle$. (By
lemma \ref{Tits}, $C$ is a subgroup of a finite group $\langle
D\rangle$ such that $D\subset S-A_0$ and $D$ commutes with $A_0$.)
\end{proposition}

\begin{proof} Let $A_0\equiv E_W(A)$. By lemma \ref{L54}
there is an infinite-length $A_0$-geodesic $r$, such that each
letter in $A_0$ appears infinitely often in $r$. The group
$\langle A_0\rangle$ contains a subgroup $A'$ which is a normal
finite-index subgroup of $B$. Let $\alpha _i$ be the initial
segment of $r$ of length $i$, and $C_i$ the $B/A'$ coset
containing $\bar \alpha_i$, the element of $W$ represented by
$\alpha _i$. Let $i$ be the first integer such that $C_i=C_j$ for
infinitely many $j$. Replace $r$ by the terminal segment of $r$
that follows $\alpha _i$. Then $r$ can be partitioned into
geodesics $(r_1,r_2,\ldots )$ such that $\bar r_i\in A'$. Hence
for any $i$ and any $b\in B$, $b\bar r_i b^{-1}\in A'\subset \langle
A_0\rangle$.

It suffices to show that $B\subset \langle A_0\rangle \times \langle
lk_2(A_0)\rangle$, since then each $b\in B$ is such that $b=xy$
with $x\in \langle A_0\rangle$ and $y\in \langle
lk_2(A_0)\rangle$. As $A_0\subset B$, $y\in B$ and so $B=\langle
A_0\rangle \times (B\cap\langle lk_2(A_0)\rangle)$. (Recall
$\langle A_0\rangle$ has finite index in $B$.)

Suppose $b$ is a shortest element of $B$ such that $b\not \in
\langle A_0\rangle \times \langle lk_2(A_0)\rangle$.  Let $\beta
$ be a geodesic for $b$.

\noindent {\bf Claim} The path $(\beta, r_1,r_2,\ldots )$ is
geodesic.

\noindent {\bf Proof:} Otherwise let $i$ be the first integer such
that $(\beta , \alpha _i)$ (recall $\alpha_i$ is the initial
segment of $r$ of length $i$) is not geodesic. Then $\bar \beta
\bar \alpha _i =\bar \gamma \bar \alpha _{i-1}$ where $\gamma$ is
obtained from $\beta$ by deleting some letter and $ (\gamma,
\alpha _{i-1})$ is geodesic. We have $\bar \gamma \bar \alpha
_{i-1}=b\bar \alpha _i$, and $\{b, \bar \alpha _{i-1},\bar \alpha
_i\}\subset B$, so $\bar \gamma \in B$.

We conclude the proof of this claim by showing: If $b$ is a shortest element of $B$ such that $b\not \in \langle A_0\cup lk_2(A_0)\rangle$ and $\beta$ is a geodesic for $b$, then a letter cannot be deleted from $\beta$ to give
a geodesic for an element of $B$. 

\noindent Otherwise, suppose $\beta =(b_1,\ldots
,b_m)$, $\gamma =(b_1,\ldots ,b_{i-1}, b_{i+1},\ldots ,b_m)$ is
geodesic, and $\bar \gamma \in B$. By the minimality hypothesis,
$\{b_1,\ldots ,b_{i-1}, b_{i+1},\ldots b_m\}\subset A_0\cup
lk_2(A_0)$. ``Sliding" $lk_2(A_0)$-letters of $\beta$ before
$b_i$ ``back" and those after $b_i$ ``forward", gives a geodesic
$(\beta _1,\beta_2,b_i,\beta _3,\beta _4)$ for $b$, with $lett
(\beta _1)\cup lett (\beta _4)\subset lk_2(A_0)$ and $lett (\beta
_2)\cup lett (\beta _3)\subset A_0$. Now, $\bar \beta _1\bar \beta
_2b_i\bar\beta _3 \bar \beta _4 \bar r_1\cdots \bar r_j \bar
\beta_ 4^{-1} \bar \beta _3^{-1}b_i\bar \beta_2 ^{-1}\bar \beta
_1^{-1}\in A' \subset \langle A_0\rangle$, for each $j$. This implies
$b_i\bar\beta _3 \bar r_1\cdots \bar r_j \bar \beta _3^{-1}b_i\in
\langle A_0\rangle$. For large $j$, $lett (\bar\beta _3 \bar
r_1\cdots \bar r_j \bar \beta _3^{-1})=A_0$. By lemma \ref{L56},
$b_i\in A_0\cup lk_2(A_0)$, and so $b\in \langle A\cup lk_2(A_0)\rangle$. This is contrary to our assumption and the claim is proved. $\square$

The same proof shows $(\beta, r_k,r_{k+1},\ldots )$ is geodesic
for all $k$.

Let $\delta_i$ be a geodesic for $b\bar r_ib^{-1}\in \langle
A_0\rangle$. Next we show $\vert \delta _i\vert =\vert r_i\vert$.
As $(\beta, r_i)$ is geodesic and $b\bar r_i=\bar
\delta_ib$, $\vert \delta _i\vert \geq \vert r_i\vert$. If
$\vert \delta _i\vert > \vert r_i\vert$ then $(\delta _i,\beta
)$ is not geodesic. Say $\delta _i=(x_1,\ldots ,x_k)$ for $x_i\in A_0$. Let $j$ be the largest integer such that $(x_j,\ldots , x_k, b_1,\ldots ,b_m)$ is not geodesic. Then $x_j$ deletes with say $b_i$ and $(x_{j+1},\ldots ,x_k,b_1,\ldots ,b_{i-1},b_{i+1},\ldots ,b_m)$ is geodesic. As 
$$x_{j+1} \ldots x_kb_1\ldots b_{i-1}b_{i+1}\ldots b_m=x_j\ldots x_kb\in B$$
the word $(b_1,\ldots ,b_{i-1},b_{i+1},\ldots ,b_m)$ is a geodesic for an element of $B$. This is impossible by the closing argument of our
claim.

Since $(\beta , r_1,\ldots ,r_i)$ is geodesic for all $i$, so is
$(\delta _1,\ldots ,\delta _i,\beta)$. Since 
$$(r_1,\ldots ,r_i,\beta ^{-1})^{-1}=(\beta, r_i^{-1},\ldots ,r_1^{-1})$$ 
the claim shows $(r_1,\ldots ,r_i,\beta
^{-1})$ is geodesic for all $i$. The proposition now follows directly from lemma \ref{Fin2}.
\end{proof}

\section{Minimal Splittings}

Recall that a subgroup $A$ of $W$ is a {\it minimal splitting
subgroup} of $W$ if $W$ splits non-trivially over $A$, and there
is no subgroup $B$ of $W$ such that $W$ splits non-trivially over
$B$, and $B\cap A$ has infinite index in $A$ and finite index in
$B$.

For a Coxeter system $(W,S)$ we defined $M(W)$ to be the
collection of minimal splitting subgroups groups of $W$.
Observe that if $W$ has more than 1-end, then each member of
$M(W)$ is a finite group. 
Define $K(W,S)$ to be the set of all subgroups of $W$ of the form $\langle A\rangle\times M$ for
$A\subset S$, and $M$ a subgroup of a finite special subgroup of
$\langle lk_2(A)\rangle$ (including when $\langle A
\rangle$ and/or $M$ is trivial). If $W$ is finitely generated, then
$K(W,S)$ is finite.

\begin{lemma} \label{L16N} 
Suppose $(W,S)$ is a finitely
generated Coxeter system and $\Lambda$ is a non-trivial reduced
graph of groups decomposition of $W$ such that each edge group of
$\Lambda$ is in $M(W)$. If $\Psi$ is a reduced visual graph of groups 
decomposition for $W$ such that each edge group of $\Psi$ is conjugate to a subgroup of $\Lambda$ then each edge group of $\Psi$ is in
$M(W)$. $\square$
\end{lemma}

\begin{lemma}\label{Minform} 
Suppose $(W,S)$ is a finitely generated Coxeter system and $G$ is a group in $M(W)$. Then $G$ is conjugate to a group in $K(W,S)$.
\end{lemma}

\begin{proof}
By theorem \ref{MT1}, there is $E\subset S$  and $w\in W$ such that $W$ splits non-trivially over $E$ and $w\langle E\rangle w^{-1}$ is conjugate to a subgroup of $G$. By the minimality of $G$, $\langle E\rangle$ has finite index in $w^{-1}Gw$ and the lemma follows from theorem \ref{Index}.
\end{proof}

\begin{example}
Consider the Coxeter system $(W,S)$ with $S=\{a,b,c,d,x,y\}$,
$m(u,v)=2$ if $u\in \{a,c,d\}$ and $v\in \{x,y\}$,
$m(a,b)=m(b,c)=2$, $m(c,d)=3$, $m(x,b)=m(y,b)=3$ and $m(x,y)=m(a,c)=m(a,d)=m(b,d)=\infty$. The group $W$ is 1-ended since no subset of $S$ separates the presentation
diagram $\Gamma (W,S)$ and also generates a finite group. The group $\langle x,c,y\rangle$ is a member of $M(W)$, since it is 2-ended and $\{x,c,y\}$ separates $\Gamma$. 
The set $\{x,y,b\}$ separates $\Gamma$, but $\langle x,b,y\rangle\not\in M(W)$ 
since $\langle x,y\rangle$ has finite index in 
$\langle x,c,y\rangle$ and infinite index in 
$\langle x,b,y\rangle$. 
Note that no subset of $\{x,b,y\}$ generates a group in $M(W)$. 

The element $cd$ conjugates $\{x,c,y\}$
to $\{x,d,y\}$. So, $\langle x,d,y\rangle \in M(W)$. Hence a
visual subgroup in $M(W)$ need not separate $\Gamma (W,S)$.
\end{example}

\begin{proposition} \label{L24N}
Suppose $(W,S)$ is a finitely generated Coxeter system and
$W=A\ast _CB$ is a non-trivial splitting of $W$. Then there
exists $D\subset S$ and $w\in W$ such that $\langle D\rangle \in
M(W)$, $D$ separates $\Gamma( W,S)$ and  $w\langle E_W(D)\rangle w^{-1}\subset C$ (so $w\langle D\rangle
w^{-1}\cap C$ has finite index in $w\langle D\rangle w^{-1}$).
Furthermore, if $C\in M(W)$ then $w\langle
E_W(D)\rangle w^{-1}$ has finite index in $C$.
\end{proposition}

\begin{proof}
The second part of this follows trivially from the definition of
$M(W)$ and theorem \ref{MT1}. Let $\Psi_1$ be a reduced visual
graph of groups decomposition for $A\ast _CB$. Each edge group of
$\Psi_1$ is a subgroup of a conjugate of $C$. Say $D_1\subset S$
and $\langle D_1\rangle$ is an edge group of $\Psi_1$. Then $W$
splits non-trivially as $\langle E_1\rangle \ast _{\langle
D_1\rangle}\langle F_1\rangle$, where $E_1\cup F_1=S$ and
$E_1\cap F_1=D_1$. If $\langle D_1\rangle$ is not in $M(W)$, there
exists $C_1$ a subgroup of $W$, such that $W$ splits
non-trivially as $A_1\ast _{C_1}B_1$ and such that $C_1\cap
\langle D_1\rangle$ has infinite index in $\langle D_1\rangle$
and finite index in $C_1$. Let $\Psi_2$ be a reduced visual
decomposition for $A_1\ast _{C_1}B_1$, and $D_2\subset S$ such
that $\langle D_2\rangle$ is an edge group of $\Psi _2$. Then a
conjugate of $\langle D_2\rangle$ is a subgroup of $C_1$, and
$W=\langle E_2\rangle \ast _{\langle D_2\rangle}\langle
F_2\rangle$, where $E_2\cup F_2=S$ and $E_2\cap F_2=D_2$. For
$i\in \{1,2\}$, $\langle D_i\rangle =\langle U_i\rangle \times
\langle V_i\rangle$ where $U_i=E_W(D_i)$ and $V_i=T_W(D_i)$ (so by
remark \ref{Split}, $U_i\cup V_i=D_i$ and $V_i$ is the (unique)
largest such subset of $D_i$ such that $\langle V_i\rangle$ is
finite).

It suffices to show that $U_2$ is a proper subset of $U_1$.
Choose $g\in W$ such that $g\langle D_2\rangle g^{-1}\subset C_1$. Then
by lemma \ref{Kil}, $g\langle D_2\rangle g^{-1}\cap \langle
D_1\rangle =d\langle K\rangle d^{-1}$ for $d\in \langle
D_1\rangle$ and $K=D_1\cap mD_2m^{-1}$ where $m$ is the minimal
length double coset representative of $\langle D_1\rangle
g\langle D_2\rangle$. Write $\langle K\rangle= \langle U_3\rangle
\times \langle V_3\rangle $ with $U_3=E_W(K)$ and $V_3=T_W(K)$. As $K\subset D_1$, $E_W(K)\subset E_W(D_1)$, so $U_3\subset U_1$. As $m^{-1}Km\subset D_2$, lemma \ref{LFin} implies $E_W(K)\subset E_W(D_2)$ so $U_3\subset U_2$. Hence $U_3\subset U_1\cap U_2$. 
Since $C_1\cap\langle D_1\rangle$ has infinite index in $\langle D_1\rangle $,  $d\langle
K\rangle d^{-1}$ has infinite index in $\langle D_1\rangle$. As $d_1\in \langle D_1\rangle$,  $\langle K\rangle$ has infinite index in $\langle D_1\rangle$.
Hence $U_3$ is a proper subset of $U_1$.

Recall that $g\langle D_2\rangle g^{-1}\subset C_1$ and $C_1\cap \langle
D_1\rangle$  has finite index in $C_1$ so $d\langle K\rangle
d^{-1}=g\langle D_2\rangle g^{-1}\cap \langle D_1\rangle$ has
finite index in $g\langle D_2\rangle g^{-1}$ and $g^{-1}d\langle U_3\rangle
d^{-1}g$ has finite index in $\langle D_2\rangle$. Thus, for $u$ the minimal
length double coset representative of $\langle D_2\rangle
g^{-1}d\langle U_3\rangle$,   $u\langle U_3\rangle u^{-1}$ has finite index in $\langle
D_2\rangle$. 

Since $E_W(U_3)=U_3$, lemma \ref{LFin} implies $U_3=uU_3u^{-1}\subset D_2$.
Hence $\langle U_3\rangle$ has finite index in $\langle
U_2\rangle$. By proposition \ref{Index}, $\langle
U_2\rangle=\langle U_3\rangle \times C$ for $C$ a finite subgroup
of $\langle lk_2(U_3)\rangle$. If $s\in U_2-U_3$ then as
$U_2\subset U_3\cup lk_2(U_3)$, $s\in lk_2(U_3)$. Hence $\langle
U_2\rangle =\langle U_3\rangle \times \langle U_2-U_3\rangle$. As
$\langle U_3\rangle$ has finite index in $\langle U_2\rangle$,
$\langle U_2-U_3\rangle$ is finite. By the definition of $U_2$,
$U_2=U_3$ and so $U_2$ is a proper subset of $U_1$.
\end{proof}

We can now easily recognize separating special subgroups in
$M(W)$.

\begin{corollary}\label{C7} 
Suppose $(W,S)$ is a Coxeter system and $C\subset S$ separates
$\Gamma (W,S)$. Then $\langle C\rangle \in M(W)$ iff there is no
$D\subset S$ such that $D$ separates $\Gamma (W,S)$ and $E_W(D)$ is a
proper subset of $E_W(C)$.
\end{corollary}

\begin{proof}
If $\langle C\rangle \in M(W)$, $D\subset S$ such that $D$ separates $\Gamma$ and $E_W(D)$ is a proper subset of $E_W(C)$, then by proposition
\ref{Index}, $\langle E_W(D)\rangle$ has infinite index in
$\langle E_W(C)\rangle$. But then $\langle D\rangle \cap \langle C\rangle$ has finite index in $\langle D\rangle$ and infinite index in $\langle C\rangle$, contrary to the assumption $\langle C\rangle \in M(W)$. 

If $\langle C\rangle \not\in M(W)$, then by proposition \ref{L24N}, there is $D\subset S$ and $w\in W$ such that $\langle D\rangle \in M(W)$, $D$ separates $\Gamma$, and $w\langle E_W(D)\rangle w^{-1}\subset \langle C\rangle$. By lemma \ref {LFin}, $E_W(D)\subset E_W(C)$. Since $\langle C\rangle\not \in M(W)$, $ E_W(D)$ is a proper subset of $E_W(C)$.
\end{proof}

\begin{theorem} \label{T1N} 
Suppose $(W,S)$ is a finitely generated Coxeter system,
$\Lambda$ is a reduced graph of groups decomposition for $W$ with
each edge group a minimal splitting subgroup of $W$, and $\Psi$ is
a reduced graph of groups decomposition of $W$ such that each vertex group of $\Psi$ is conjugate to a subgroup of a vertex group of $\Lambda$ and for each edge $E$ of $\Lambda$, there is an edge $D$ of $\Psi$ such that $\Psi(D)$ is conjugate to a subgroup of $\Lambda(E)$. (E.g. if $\Psi$ is a visual decomposition from $\Lambda$.)
If $A$ is a vertex of
$\Lambda$, and $\Phi _A$ is the reduced decomposition of $\Lambda
(A)$ given by the action of $\Lambda (A)$ on the Bass-Serre tree
for $\Psi$, then

1) For each edge $E$ of $\Lambda$ adjacent to $A$, $\Lambda
(E)\subset a\Phi_A(K)a^{-1}$, for some $a\in \Lambda (A)$ and
some vertex $K$ of $\Phi_A$. In particular, the decomposition
$\Phi_A$ is compatible with $\Lambda$.

2) Each vertex group of $\Phi_A$ is conjugate to a vertex group of
$\Psi$ (and so is Coxeter), or is $\Lambda (A)$-conjugate to
$\Lambda (E)$ for some edge $E$ adjacent to $A$.

3) If each edge group of $\Psi$ is in $M(W)$, then each edge group of $\Phi_A$ is a minimal splitting subgroup of $W$.
\end{theorem}

\begin{proof}
Suppose $E$ is an edge of $\Lambda$ adjacent to $A$. 
By hypothesis, there is an edge $D$ of $\Psi$ and $w\in W$
such that $w\Psi(D)w^{-1}\subset \Lambda (E)$. Since $\Lambda (E)$
is minimal, $\Psi(D)$ has finite index in $w^{-1}\Lambda(E)w$ and
so corollary 4.8 of \cite{DicksDunwoody} implies $\Lambda (E)$
stabilizes a vertex of $T_{\Psi}$, the Bass-Serre tree for
$\Psi$. Thus $\Lambda (E)$ is a subgroup of $a\Phi_A(K)a^{-1}$,
for some vertex $K$ of $\Phi _A$, and some $a\in \Lambda(A)$. Part
1) is proved.

By theorem \ref{artificial}, each vertex group of $\Phi_A$ is either
conjugate to a vertex group of $\Psi$ or $\Lambda(A)$-conjugate to
a subgroup of an edge group $\Lambda (E)$, for some edge $E$ of
$\Lambda$ adjacent to $A$. Suppose $Q$ is a vertex of $\Phi_A$ and
$a_1\Phi_A(Q)a_1^{-1}\subset \Lambda (E)$ for some $a_1\in
\Lambda (A)$. By part 1), $\Lambda (E)\subset
a_2\Phi_A(K)a_2^{-1}$, for some $a_2\in \Lambda(A)$ and $K$ a
vertex of $\Phi_A$. Thus, $a_1\Phi_A(Q)a_1^{-1}\subset \Lambda
(E)\subset a_2\Phi_A(K)a_2^{-1}$.  Lemma \ref{Sub} implies $Q=K$ and $a_2^{-1}a_1\in \Phi_A(Q)$, so
$\Phi_A(Q)=a_2^{-1}\Lambda (E)a_2$ and part 2) is proved.

By part 1) $W$ splits non-trivially over each edge group of $\Phi_A$ and part 3) follows.
\end{proof}

\begin{proposition} \label{P2N} 
Suppose $(W,S)$ is a finitely generated Coxeter system, $\Lambda$ is a reduced graph of groups decomposition of $W$ and $E$ is an edge of $\Lambda$ such that $\Lambda (E)$ is conjugate to a group in $K(W,S)$. Then there is $Q\subset S$ such that a conjugate of $\langle Q\rangle$ is a subgroup of a vertex group of $\Lambda$ and a conjugate of $\Lambda (E)$ has finite index in $\langle Q\rangle$. \end{proposition}

\begin{proof} The group $\Lambda (E)$ is conjuate to $\langle B\rangle \times F$ for $B\subset S$ and $F\subset \langle D\rangle$ where $D\subset lk_2(B)$ and $\langle D\rangle$ is finite. Let $T_{\Lambda}$ be the Bass-Serre tree for $\Lambda$ and set $B=\{b_1,\ldots ,b_n\}$. It suffices to show that $\langle B\cup D\rangle$ stabilizes a vertex of $T_{\Lambda}$.  Otherwise, let $i\in \{0,1,\ldots , n-1\}$ be large as possible so that $\langle D\cup \{b_1,\ldots ,b_i\}\rangle$ stabilizes a vertex of $T_{\Lambda}$. As $\langle D\cup \{b_{i+1}\}\rangle$ is finite, it stabilizes some vertex $V_1$ of $T_{\Lambda}$. The group $\langle B\rangle$ stabilizes a vertex $V_2$ of $T_{\Lambda}$ and $\langle D\cup \{b_1,\ldots ,b_{i}\}\rangle$ stabilizes a vertex $V_3$ of $T_{\Lambda}$. Since $T_{\Lambda}$ is a tree, there is a vertex $V_4$ common to the three $T_{\Lambda}$-geodesics connecting pairs of vertices in $\{V_1,V_2,V_3\}$. Then $\langle D\cup \{b_1,\ldots ,b_{i+1}\rangle $ stabilizes $V_4$, contrary to the minimality of $i$. Instead, $\langle D\cup B\rangle$ stabilizes a vertex of $T_{\Lambda}$. 
\end{proof}

The next result combines theorem \ref{T1N} and proposition \ref{P2N} to show that any graph of groups decomposition of a Coxeter group with edge groups equal to minimal splitting subgroups of the Coxeter group is, up to ``artificial considerations", visual.

\begin{proposition} \label{P3N} 
Suppose $(W,S)$ is a finitely generated Coxeter system, $\Lambda$ is a reduced graph of groups decomposition for $W$ with each edge group a minimal splitting subgroup of $W$, and $\Psi$ is a reduced visual decomposition from $\Lambda$. If $\Phi'$ is the graph of groups obtained from $\Lambda$ by replacing each vertex $A$ of $\Lambda$ by $\Phi_A$, the graph of groups decomposition of $\Lambda(A)$ given by the action of $\Lambda(A)$ on the Bass-Serre tree for $\Psi$, and $\Phi$ is obtained by reducing $\Phi'$, then there is a bijection $\tau$, from the vertices of $\Phi$ to those of $\Psi$ so that for each vertex $V$ of $\Phi$, $\Psi(\tau (V))$ is conjugate to $\Phi(V)$.
\end{proposition}

\begin{proof} Part 1) of theorem \ref {T1N} implies the decomposition $\Phi$ is well-defined. If $Q$ is a vertex of $\Psi$ then a conjugate of $\Psi(Q)$ is a subgroup of  $\Lambda(B)$ for some vertex $B$ of $\Lambda$, and corollary 7 of \cite{MTVisual} (an elementary corollary of theorem \ref{artificial}) implies this conjugate of $\Psi(Q)$ is a vertex group of $\Phi_B$. Hence each vertex group of $\Psi$ is conjugate to a vertex group of $\Phi'$. Suppose $A$  is a vertex of $\Lambda$ and $U$ is a vertex of $\Phi_A$ such that $\Phi_A(U)$ is $\Lambda (A)$-conjugate to $\Lambda (E)$ for some edge $E$ adjacent to $A$. If $\Lambda (E)$ is not conjugate to a special subgroup of $(W,S)$, then as $\Lambda (E)$ is conjugate to a group in $K(W,S)$, proposition \ref{P2N} implies there is a vertex $V$ of $\Lambda$ and a vertex group of $\Phi_V$ properly containing a conjugate of $\Lambda (E)$. Hence $\Phi_A(U)$ is eliminated by reduction when $\Phi$ is formed. If $\Lambda (E)$ is conjugate to a special subgroup of $(W,S)$, then as $\Lambda (E)$ is also conjugate to a subgroup of a vertex group of $\Psi$, either $\Lambda(E)$ is conjugate to a vertex group of $\Psi$ or $\Lambda(E)$ is eliminated by reduction when $\Phi$ is formed. Hence by part 2) of theorem \ref{T1N}, every vertex group of $\Phi$ is conjugate to a vertex group of $\Psi$. No two vertex groups of $\Psi$ are conjugate, so if $V$ is a vertex of $\Phi$, let $\tau (V)$ be the unique vertex of $\Psi$ such that $\Phi(V)$ is conjugate to $\Psi(\tau(V))$. As no two vertex groups of $\Phi$ are conjugate, $\tau$ is injective. If $Q$ is a vertex of $\Psi$, then as noted above $\Psi(Q)$ is conjugate to a vertex group of $\Phi'$ and so $\Psi(Q) \subset w\Phi(V)w^{-1}$ for some $w\in W$ and $V$ a vertex of $\Phi$. Choose $x\in W$ such that $\Phi(V)=x\Psi(\tau(V))x^{-1}$. Then $\Psi(Q)\subset wx\Psi(\tau(V))x^{-1}w^{-1}$. Lemma \ref{Sub} implies $Q=\tau(V)$ and so $\tau$ is onto.
\end{proof}

In the previous argument it is natural to wonder if a vertex group of $\Psi$ might be conjugate to a vertex group of $\Phi_A$ and to a vertex group of $\Phi_B$ for $A$ and $B$ distinct vertices of $\Lambda$. Certainly such a group would be conjugate to an edge group of $\Lambda$. The next example show this can indeed occur.

\begin{example}
Consider the Coxeter presentation $\langle a,b,c,d : \  a^2=b^2=c^2=d^2=1\rangle$. Define $\Lambda$ to be the graph of groups decomposition $\langle a,cdc\rangle\ast_{\langle cdc\rangle}\langle b,cdc\rangle\ast \langle d\rangle$. Then $\Lambda$ has graph with a vertex $A$ and $\Lambda (A)=\langle a,cdc\rangle$, edge $C$ with $\Lambda(C)=\langle cdc\rangle$ vertex $B$ with $\Lambda(B)=\langle b,cdc\rangle$ edge $E$ with $\Lambda(E)$ trivial and vertex $D$ with $\Lambda(D)=\langle d\rangle$.
The visual decomposition for $\Lambda$ is $\Psi=\langle a\rangle\ast \langle b\rangle\ast \langle c\rangle\ast \langle d\rangle$, a graph of groups decomposition with each vertex group isomorphic to $\mathbb Z_2$ and each edge group trivial. Now $\Phi_A$ has decomposition $\langle a\rangle\ast \langle cdc\rangle$, $\Phi_B$ has decomposition $\langle b\rangle \ast \langle cdc\rangle$  and $\Phi_D$ has decomposition $\langle d\rangle$. Observe that the $\Psi$ vertex group $\langle d\rangle$ is conjugate to a vertex group of both $\Phi_A$ and $\Phi_B$. The group $\Phi$ of the previous theorem would have decomposition $\langle a\rangle\ast \langle b\rangle\ast \langle c\rangle \ast \langle cdc\rangle$.  
\end{example}

\begin{lemma} \label{Kequal} 
Suppose $(W,S)$ is a finitely generated Coxeter system and $C$ is  a subgroup of $W$ conjugate to a group in $K(W,S)$. If $D$ is a subgroup of $W$ and $wDw^{-1}\subset C\subset D$ for some $w\in W$, then $wDw^{-1}=C=D$.
\end{lemma}

\begin{proof} Conjugating we may assume $C=\langle U\rangle\times F$, for $U\subset S$, $E_W(U)=U$ and $F$ a finite group.  Let $K\subset lk_2(U)$ such that $\langle K\rangle$ is finite and $F\subset \langle K\rangle$. Now, $w\langle U\rangle w^{-1}\subset wCw^{-1}\subset wDw^{-1}\subset C\subset \langle U\cup K\rangle$. Write $w=xdy$ for $x\in \langle U\cup K\rangle$, $y\in \langle U\rangle$, and $d$ the minimal length double coset representative of $\langle U\cup K\rangle w\langle U\rangle$.  Then $dCd^{-1}\subset dDd^{-1}\subset x^{-1}Cx$. By lemma \ref{LFin},  $dUd^{-1}=U$ and by the definition of $x$, $x^{-1}\langle U\rangle x=\langle U\rangle$.  The index of $\langle U\rangle$ in $dCd^{-1}$ is $\vert F\vert $ and the index of $\langle U\rangle$ in $x^{-1}Cx$ is $\vert F\vert$. Hence $dCd^{-1}=dDd^{-1}=x^{-1}Cx$ and $wCw^{-1}=wDw^{-1}=C$. 
\end{proof}

\begin{remark} 
The argument in the first paragraph below shows that if $\Lambda$ is a reduced graph of groups decomposition of a Coxeter group $W$, $V$ is a vertex of $\Lambda$ and $\Phi$ is a reduced graph of groups decomposition of $\Lambda (V)$, compatible with $\Lambda$ then when replacing $V$ by $\Phi$ to form $\Lambda_1$, no vertex group of $\Phi$ is $W$-conjugate to a subgroup of another vertex group of $\Phi$. In particular, each edge of $\Phi$ survives reduction in $\Lambda_1$.
\end{remark}

\begin{proposition} \label{Kreduce} 
Suppose $(W,S)$ is a finitely generated Coxeter system and $\Lambda$ is a reduced graph of groups decomposition of $W$ with $M(W)$ edge groups. Suppose a vertex group of $\Lambda$ splits nontriviall and compatibly as $A\ast _CB$ over an $M(W)$ group $C$. Then there is a group in $ K(W,S)$ contained in a conjugate of $B$ which is not also contained in a conjugate of $A$ (and then also with $A$ and $B$ reversed).
\end{proposition}
\begin{proof} Let $V$ be the vertex group such that $\Lambda(V)$ splits as $A\ast _CB$ and let $\Lambda_1$ be the graph of groups resulting from replacing $\Lambda(V)$ by this splitting. If there is $w\in W$ such that $wBw^{-1}\subset A$, then (by considering the Bass-Serre tree for $\Lambda_1$) a $W$-conjugate of $B$ is a subgroup of $C$. Lemma \ref {Kequal} then implies $B=C$, which is nonsense. Hence no $W$-conjugate of $B$ (respectively $A$) is a subgroup of $A$ (respectively $B$).  This implies that if $\Lambda_2$ is obtained by reducing $\Lambda_1$, then there is an edge $\bar C$ of $\Lambda_2$ with vertices $\bar A$ and $\bar B$, such that $\Lambda_2(\bar C)=C$, and $\Lambda_2(\bar A)$ is $\hat A$ where $\hat A$ is either $A$ or a vertex group (other than $\Lambda_1(V)$) of $\Lambda_1$ containing $A$ as a subgroup. Similarly for $\Lambda _2(\bar B)$.

If $B$ collapses across an edge of $\Lambda_1$ then $B$ is conjugate to a group in $K(W,S)$ and $B$ satisfies the conclusion of the proposition. If $B$ does not collapse across an edge of $\Lambda_1$ (so that $\hat B=B$), then let $\Phi_B$ be the reduced graph of groups decomposition of $B$ induced from the action of $B$ on $\Psi$, the visual decomposition of $W$ from $\Lambda_2$. By theorem \ref{T1N}, each vertex group of $\Phi_B $ is conjugate to a group in $K(W,S)$ and the decomposition $\Phi_B$ is compatible with $\Lambda_2$. Let $\Lambda_3$ be the graph of groups decomposition of $W$ obtained from $\Lambda_2$ by replacing the vertex for $B$ by $\Phi_B$.  In $\Lambda_3$, the edge $\bar C$ connects the vertex $\bar A$ to say the $\Phi_B$-vertex $\tilde B$. If $\Lambda _3(\tilde B)$ is not conjugate to a subgroup of $A$, then $\Lambda _3(\tilde B)$ satisfies the conclusion of our proposition. Otherwise, (as before) lemma \ref{Kequal} implies $\Lambda_3(\bar C)=\Lambda_3(\tilde B)$ and we collapse $\tilde B$ across $\bar C$ to form $\Lambda _4$. Note that if $\bar C$ does collapse, then $\Phi_B$ has more than one vertex.  
There is an edge of $\Lambda_4$ (with edge group some subgroup of $C$ which is also an edge group of $\Phi_B$) separating the vertex $\bar A$ from some vertex $K$ of $\Phi_B$. The group $\Lambda_4(K)$ satisfies the conclusion of the proposition, since otherwise a $W$-conjugate of $\Lambda _4(K)$ is a subgroup of $A$. But then lemma \ref{Kequal} implies $\Lambda _4(K)$ is equal to an edge group of $\Phi_B$ which is impossible. 
\end{proof}

Proposition \ref{Kreduce} is the last result of this section needed to prove our main theorem. The remainder of the section is devoted to proving theorem \ref{vismin}, a minimal splitting version of the visual decomposition theorem of \cite{MTVisual}. In order to separate this part of the paper from the rest, some lemmas are listed here that could have been presented in earlier sections. 
The next lemma follows directly from theorem \ref{Index}.

\begin{lemma} \label{properE}
Suppose $(W,S)$ is a finitely generated Coxeter system and
$A\subset S$. If $B$ is a proper subset of $E(A)$ then $\langle
B\rangle$ has infinite index in $\langle E(A)\rangle$. $\square$
\end{lemma}

\begin{lemma}\label{minequiv} 
Suppose $(W,S)$ is a finitely generated Coxeter system, $A$ and
$B$ are subsets of $S$ such that $\langle A\rangle$ and $\langle
B\rangle$ are elements of $M(W)$. If $E(A)\subset B$ then
$E(A)=E(B)$.
\end{lemma}

\begin{proof} If $E(A)\subset B$, then the definitions of $E(A)$
and $E(B)$, imply $E(A)\subset E(B)$. As $\langle B\rangle \in
M(W)$, lemma \ref{properE} implies $E(A)$ is not a proper subset
of $E(B)$.
\end{proof}

\begin{lemma} \label{connect} 
Suppose $(W,S)$ is a finitely generated Coxeter system, $C\subset
S$ is such that $\langle C\rangle \in M(W)$ and $C$ separates
$\Gamma (W,S)$. If $K\subset S$ is a component of $\Gamma -C$,
then for each $c\in E(C)$, there is an edge connecting $c$ to $K$.
\end{lemma}

\begin{proof} Otherwise, $C-\{c\}$ separates $\Gamma$. This is
impossible by lemma \ref{properE} and the fact that $\langle
C\rangle \in M(W)$.
\end{proof}

In the remainder of this section we simplify notation for visual graph of groups decompositions by labeling each vertex of such a graph by $A$, where $A\subset S$ and $\langle A\rangle$ is the vertex group. It is possible for two distinct edges of such a decomposition to have the same edge groups so we do not extend this labeling to edges.

\begin{lemma} \label{mincompat2}
Suppose $(W,S)$ is a finitely generated Coxeter system and $\Psi$
is a reduced $(W,S)$-visual graph of groups decomposition with
$M(W)$-edge groups. If $A\subset S$ is a vertex of $\Psi$, and
$M\subset S$ is such that $\langle M\rangle \in M(W)$, $M$
separates $\Gamma(W,S)$ and $E(M)\subset A$, then

1) either $E(M)=E(C)$ for some $C\subset S$ and $\langle C\rangle$ the edge group of an edge of $\Psi$ adjacent to $A$, or $M\subset A$ and $M$ separates $A$ in
$\Gamma$, and

2) for each $C\subset S$ such that $\langle C\rangle$ is the edge group of an edge of $\Psi$ adjacent to $A$, $C-M$ is
a subset of a component of $\Gamma -M$.

In particular, if $E(M)\ne E(C)$ for each $C\subset S$ such that $\langle C\rangle$ is the edge group of an edge adjacent to $A$ in $\Psi$, then $\langle A\rangle$ visually splits over
$\langle M\rangle$, compatibly with $\Psi$, such that each vertex
group of the splitting is generated by $M$ union the intersection of $A$ with a
component of $\Gamma -M$.
\end{lemma}

\begin{proof}  First we show that if $M\not \subset A$, then
$E(M)=E(C)$ for some $C$ such that $\langle C\rangle$ is the edge group of an edge adjacent to $A$ in $\Psi$. If
$E(M)=\emptyset$ then $\langle M\rangle$ is finite and
$E(C)=\emptyset$ for every $C\subset S$ such that $\langle
C\rangle \in M(W)$. Hence we may assume $E(M)\ne \emptyset$. As
$E(M)\subset A$, there is $m\in M-E(M)$ such that $m\not \in A$.
Say $m\in B$ for $B\subset S$ a vertex of $\Psi$. If $E$
is the first edge of the $\Psi$-geodesic from $A$ to $B$ and $\Psi (E)=C$, then
$m\not \in C$. But in $\Gamma$, there is an edge between $m$ and
each vertex of $E(M)$. Hence $E(M)\subset C$ and lemma
\ref{minequiv} implies $E(M)=E(C)$.

To complete part 1), it suffices to show that if $E(M)\ne E(C)$
for all $C\subset S$ such that $\langle C\rangle$ is the edge group of an edge of $\Psi$ adjacent to the vertex $A$
of $\Psi$, then $M$ separates $A$ in $\Gamma$. We have shown that
$M\subset A$. Write $W=\langle D_C\rangle \ast _{\langle C\rangle}
\langle B_C\rangle$ where $C\subset S$ is such that $\langle C\rangle=\Psi(E)$ for $E$ an edge of $\Psi$ adjacent to $A$, and $B_C$ (respectively $D_C$)  the union of
the $S$-generators of vertex groups for all vertices of $\Psi$ on
the side of $E$ opposite $A$ (respectively, on the same side of
$C$ as $A$). In particular, $M\subset D_C$ and $M\cap
(B_C-C)=\emptyset$. Then $B_C$ is the union of $C$ and some of the
components of $\Gamma -C$ (and $D_C$ is the union of $C$ and the
rest of the components of $\Gamma -C$). By lemma \ref{minequiv},
$E(C)\not \subset M$. Choose $c\in E(C)-M$. If $B'$ is a
component of $\Gamma -C$ and $B'\subset B_C$, then by lemma
\ref{connect}, there is an edge of $\Gamma$ connecting $c$ and
$B'$. Hence $(B_C-C)\cup (E(C)-M)\subset K_C$ for some component
$K_C$ of $\Gamma -M$. In particular, $c\in A\cap K_C$. Also note
that if $c'\in C-M$ then either $c'\in E(C)-M$ or there is an edge
of $\Gamma$ connecting $c'$ to $c$. In either case $c'\in K_C$
and $(B_C-C)\cup (C-M)\subset K_C$.

For $i\in \{1,\ldots ,n\}$, let $E_i$ be the edges of
$\Psi$ adjacent to $A$ and let $\langle C_i\rangle=\Psi(E_i)$ for $C_i\subset S$. Since $\langle A\rangle$ is a vertex group of $\Psi$, $\Gamma -A=\cup _{i=1}^n
(B_{C_i}-C_i)\subset \cup_{i=1}^n K_{C_i}$. We have argued that
there is $c_i\in A\cap K_{C_i}$ for the  component $K_{C_i}$ of
$\Gamma -M$. If $K_{C_i}\ne K_{C_j}$, then $M$ separates the
points $c_i$ and $c_j$ of $A$, in $\Gamma$. If all $K_{C_i}$ are
equal (e.g. when $n=1$), then $\Gamma -K_{C_i}\subset A$. Since
$M$ separates $\Gamma$, $\Gamma \ne K_{C_i}\cup M$, so $M$
separates $c_i$ from a point of $A-(K_{C_i}\cup M)$. In any case
part 1) is proved.

Part 2): As noted above, if $E(M)\ne E(C)$, then for any
$C\subset S$  such that $\langle C\rangle$ is the edge group of an edge of $\Psi$ adjacent to $A$ we have $(B_C-C)\cup
(C-M)\subset K_C$ for $K_C$ a component of $\Gamma -M$ and $B_C$
some subset of $S$. If $E(M)=E(C)$ then $\langle C-M\rangle$ is finite, so $C-M$ is a complete
subset of $\Gamma$ and hence a subset of a component of $\Gamma
-M$.
\end{proof}

The next result is a minimal splitting version of the visual decomposition theorem. While part 2) of the conclusion is slightly weaker than the corresponding conclusion of the visual decomposition theorem, part 3) ensures that all edge groups of a given graph of groups decomposition of a finitely generated Coxeter group are ``refined" by minimal visual edge groups of a visual decomposition. The example following the proof of this theorem shows that part 2) cannot be strengthened. 

\begin{theorem} \label{vismin} 
Suppose $(W,S)$ is a finitely generated Coxeter system and
$\Lambda$ is a reduced graph of groups decomposition for $W$.
There is a reduced visual decomposition $\Psi$ of $W$ such
that

1) each vertex group of $\Psi$ is a subgroup of a conjugate of a
vertex group of $\Lambda$,

2) if $D$ is an edge of $\Psi$ then either $\Psi (D)$ is
conjugate to a subgroup of an edge group of $\Lambda$, or
$\Psi(D)$ is a minimal splitting subgroup for $W$ and a visual subgroup of finite index in
$\Psi (D)$ is conjugate to a subgroup of an edge group of
$\Lambda$.

3) for each edge $E$ of $\Lambda$ there is an edge $D$ of $\Psi$
such that $\Psi(D)$ is a minimal splitting subgroup for $W$, and a visual subgroup of finite
index of $\Psi (D)$ is conjugate to a subgroup of $\Lambda (E)$.

\end{theorem}

\begin{proof} Let $C_1$ be an edge group of $\Lambda$. By
proposition \ref{L24N} there exists $M_1\subset S$ and $w\in W$ such
that $\langle M_1\rangle \in M(W)$, $M_1$ separates $\Gamma
(W,S)$ and $w\langle M_1\rangle w^{-1}\cap C_1$ has finite index
in $w\langle M_1\rangle w^{-1}$. Then $W$ visually splits as
$\Psi_1\equiv \langle A_1\rangle \ast _{\langle M_1\rangle
}\langle B_1\rangle$ (so $A_1\cup B_1=S$, $M_1=A_1\cap B_1$, and
$A_1$ is the union of $M_1$ and some of the components of
$\Gamma-M_1$ and $B_1$ is $M_1$ union the other components of
$\Gamma -M_1$). Suppose $C_2$ is an edge group of $\Lambda$ other
than $C_1$. Then $W=K_2\ast _{C_2}L_2$ where $K_2$ and $L_2$ are
the subgroups of $W$ generated by the vertex groups of $\Lambda$
on opposite sides of $C_2$. Let $T_2$ be the Bass-Serre tree for
this splitting.

Suppose $\langle A_1\rangle$ and $\langle B_1\rangle$ stabilize
the vertices $X_1$ and $Y_1$ respectively of $T_2$. Then $X_1\ne
Y_1$, since $W$ is not a subgroup of a conjugate of $K_2$ or
$L_2$. Now, $\langle M_1\rangle$ stabilizes the $T_2$-geodesic
connecting $X_1$ and $Y_1$ and so $\langle M_1\rangle$ is a
subgroup of a conjugate of $C_2$. In this case we define
$\Psi_2\equiv \Psi_1$.

If $\langle A_1\rangle$ does not stabilize a vertex of $T_2$ then
there is a non-trivial visual decomposition $\Phi_1$ of $\langle
A_1\rangle$ from its action on $T_2$ as given by the visual
decomposition theorem. Since a conjugate of $\langle M_1\rangle
\cap w^{-1}C_1w$ has finite index in $\langle M_1\rangle$ and at
the same time stabilizes a conjugate of a vertex group of
$\Lambda$ (and hence a vertex of $T_2$), corollary 4.8 of
\cite{DicksDunwoody} implies $\langle M_1\rangle$ stabilizes a vertex
of $T_2$, and so $\Phi_1$ is visually compatible with the visual
splitting $\Psi_1=\langle A_1\rangle \ast _{\langle M_1\rangle }\langle
B_1\rangle$. If $\langle E_2\rangle$ is an edge group of $\Phi_1$,
then a conjugate of $\langle E_2\rangle $ is a subgroup of $C_2$.
By corollary \ref{C7}, there is $M_2\subset S$ such that $M_2$
separates $\Gamma(W,S)$,  $\langle M_2\rangle\in M(W)$ and
$E(M_2)\subset E_2$ and so $\langle E(M_2)\rangle$ is a subgroup
of a conjugate of $C_2$. If $E(M_2)\ne E(M_1)$, then lemma
\ref{mincompat2} implies $M_2\subset A_1$ and $\langle
A_1\rangle$ visually splits over $\langle M_2\rangle$ compatibly
with the splitting $\Psi_1$. Reducing produces a visual
decomposition $\Psi_2$. Similarly if $\langle A_1\rangle$
stabilizes a vertex of $T_2$ and $\langle B_1\rangle$ does not.

Inductively, assume $C_1,\ldots ,C_n$ are distinct edge groups of
$\Lambda$, $\Psi_{n-1}$ is a reduced visual graph of
groups decomposition, each edge group of $\Psi_{n-1}$ is in
$M(W)$ and contains a visual subgroup of finite index conjugate to a
subgroup of $C_i$ for some $1\leq i\leq n-1$, and for each $i\in
\{1,2,\ldots ,n-1\}$ there is an edge group $\langle M_i\rangle$ $(M_i \subset S)$ of
$\Psi_{n-1}$ such that a visual subgroup of finite index of $\langle M_i\rangle$ is
conjugate to a subgroup of $C_i$. Write $W=K_n\ast _{C_n}L_n$ as
above, and let $T_n$ be the Bass-Serre tree for this splitting.
Either two adjacent vertex groups of $\Psi_{n-1}$ stabilize
distinct vertices of $T_n$ (in which case we define $\Psi
_n\equiv \Psi_{n-1}$) or some vertex $V_i\subset S$ of $\Psi
_{n-1}$ does not stabilize a vertex of $T_n$. In the latter case
$\langle V_i\rangle$ visually splits (as above) to give $\Psi_n$.
Hence, we obtain a reduced visual decomposition $\Psi'$
such that for each edge group $\langle M\rangle$ $(M\subset S)$ of $\Psi'$, $\langle
M\rangle$ is a group in $M(W)$, a subgroup of finite index in $\langle
M\rangle$ is conjugate to a subgroup of an edge group of
$\Lambda$, and for each edge $D$ of $\Lambda$ there is
an edge group $\langle M\rangle$ of $\Psi'$ such that $\langle
E(M)\rangle$ (a subgroup of finite index in $\langle M\rangle$)
is conjugate to a subgroup of $\Lambda (D)$.

Suppose $V\subset S$ is a vertex of $\Psi'$. Consider $\Phi_V$,
the visual decomposition of $\langle V\rangle$ from its action on
$T_{\Lambda}$, the Bass-Serre tree for $\Lambda$. If $\langle D\rangle $ $(D\subset S)$
is an edge group for an edge of $\Psi '$ adjacent to $V$, then a subgroup of finite
index in $\langle D\rangle$  stabilizes a vertex of $T_{\Lambda}$.
By corollary 4.8 of \cite{DicksDunwoody}, $\langle D\rangle$
stabilizes a vertex of $T_{\Lambda}$ and $\Phi _V$ is compatible
with $\Psi'$. Replacing each vertex $V$ of $\Psi'$ by $\Phi_V$
and reducing gives the desired decomposition of $W$.
\end{proof}

The following example exhibits why one cannot expect a stronger
version of theorem \ref{vismin} with visual decomposition $\Psi$
having only minimal edge groups, or so that all minimal edge
groups of $\Psi$ are conjugate to subgroups of edge groups of
$\Lambda$.

\begin{example}
Consider the Coxeter presentation $\langle a_1, a_2, a_3, a_4,
a_5: a_i^2=1,
(a_1a_2)^2=(a_2a_3)^2=(a_3a_4)^2=(a_4a_5)^2=(a_5a_1^2)=(a_2a_5)^2=1\rangle$
and the splitting $\Lambda=\langle a_2,a_3,a_4\rangle \ast
_{\langle a_2,a_4\rangle }\langle a_1,a_2, a_4, a_5\rangle$. The
subgroup $\langle a_2,a_5\rangle$ is the only minimal visual
splitting subgroup for this system, and it is smaller than
$\langle a_2,a_4\rangle$. Then no subgroup of $\langle
a_2,a_4\rangle$ is a minimal splitting subgroup for our group.
The only visual decomposition for this splitting satisfying the
conclusion of theorem \ref{vismin} is: $\langle
a_1,a_2,a_5\rangle \ast _{\langle a_2,a_5\rangle}\langle
a_2,a_4,a_5\rangle \ast _{\langle a_2,a_4\rangle} \langle
a_2,a_3,a_4\rangle$.

\end{example}

\section{Accessibility}

We prove prove our main theorem in this section, a strong
accessibility result for splittings of Coxeter groups over groups
in $M(W)$. For a class of groups $\cal V$, we call a graph of groups decomposition of a
group {\it irreducible with respect to $\cal V$-splittings} if
for any vertex group $V$ of the decomposition, every non-trivial
splitting of $V$ over a group in $\cal V$ is not compatible with
the original graph of groups decomposition.

The following simple example describes a non-trivial compatible
splitting of a vertex group of a graph of groups decomposition
$\Lambda$, of a Coxeter group followed by a reduction to produce
a graph of groups with fewer edges than those of $\Lambda$. This
illustrates potential differences between accessibility and strong
accessibility.

\begin{example}
$$W\equiv \langle s_1,s_2:s_i^2\rangle\times \langle
s_3,s_4,s_5,s_6:s_i^2\rangle$$

First consider the splitting of $W$ as:

$$\langle s_1,s_2,s_3,s_4\rangle \ast _{\langle
s_1,s_2,s_4\rangle}\langle s_1,s_2,s_4,s_5\rangle \ast _{\langle
s_1,s_2,s_5\rangle}\langle s_1,s_2, s_5, s_6\rangle$$

The group $\langle s_1,s_2, s_4,s_5\rangle$ splits as $\langle
s_1,s_2,s_4\rangle \ast _{\langle s_1,s_2\rangle} \langle
s_1,s_2,s_5\rangle$. Replacing this group in the above splitting
with this amalgamated product and collapsing gives the following
decomposition of $W$:

$$\langle s_1,s_2,s_3,s_4\rangle \ast _{\langle
s_1,s_2\rangle}\langle s_1,s_2,s_5,s_6\rangle$$
\end{example}

\begin{proposition}\label{twoendvisualsplit} 
Suppose $(W,S)$ is a finitely generated Coxeter system, $\Psi$ is
a reduced visual graph of groups decomposition of $(W,S)$, with
$M(W)$ edge groups and $V$ is a vertex of $\Psi$ such that
$\Psi (V)$ decomposes compatibly as a nontrivial amalgamated product
$A*_{C}B$ where $C$ is in $M(W)$. Then $\Psi (V)$ is a nontrivial
amalgamated product of special subgroups over an $M(W)$ special
subgroup $U$, with $U$ a subgroup of a conjugate of $C$, and such
that any special subgroup contained in a conjugate of $A$ or $B$
is a subgroup of one of the factors of this visual splitting. In
particular, the vertex group $\Psi (V)$ visually splits, compatibly with
$\Psi$, to give a finer visual decomposition of $(W,S)$.
\end{proposition}

\begin{proof}
Applying theorem \ref{MT1} to the
amalgamated product $A \ast_CB$, we get that there is a reduced
visual graph of groups decomposition $\Psi '$ of $\Psi (V)$ such that each vertex
group of $\Psi '$ is a subgroup of a conjugate of $A$ or $B$ and each edge group
a subgroup of a conjugate of $C$. Then $\Psi '$ has more than one
vertex since $A*_{C}B$ being nontrivial means $\Psi (V)$ is not a
subgroup of a conjugate of $A$ or $B$.  Fix an edge of $\Psi '$,
say with edge group $U$, and collapse the other edges in $\Psi '$
to get a nontrivial visual splitting of $\Psi (V)$ over $U$ a subgroup of a
conjugate of $C$. By  theorem \ref{MT1},
a special subgroup of $\Psi (V)$ contained in a conjugate of $A$ or $B$
is contained in a vertex group of $\Psi '$ and so is contained in
one of the factors of the resulting visual splitting of $\Psi (V)$
derived from partially collapsing $\Psi '$. Hence this visual
decomposition of $\Psi (V)$ is compatible with $\Psi$, giving a finer
visual decomposition of $(W,S)$. Since $C$ is in $M(W)$ and a
conjugate of $U$ is a subgroup of $C$, $U$ is in $M(W)$.
\end{proof}

A visual decomposition $\Psi$ of a  Coxeter system $(W,S)$
{\it looks irreducible with respect to $M(W)$ splittings} if each
edge group of $\Psi$ is in $M(W)$ and for any subset $V$ of $S$ such that $\langle V\rangle$ is a vertex group of $\Psi$, $\langle V\rangle$ cannot be split visually, non-trivially and $\Psi$-compatibly over $\langle E\rangle\in M(W)$ for $E\subset S$, to give a finer visual decomposition of $W$. By lemma \ref{MT3}, it is
elementary to see that every finitely generated Coxeter group has
a visual decomposition that looks irreducible with respect to
$M(W)$ splittings. The following result is a direct consequence of
Proposition \ref{twoendvisualsplit}.

\begin{corollary}
A visual decomposition of a Coxeter group looks irreducible with
respect to $M(W)$ splittings, iff it is irreducible with respect
to $M(W)$ splittings. $\square$
\end{corollary}

Hence any visual graph of groups decomposition of a Coxeter group
with $M(W)$ edge groups can be refined to a visual decomposition
that is irreducible with respect to $M(W)$ splittings.

\begin{corollary}\label{JSJproposition}
Suppose $(W,S)$ is a finitely generated Coxeter system and $W$ is
the fundamental group of a graph of groups $\Lambda$ where each
edge group is in $M(W)$.  Then $W$ has an irreducible with respect
to $M(W)$ splittings visual decomposition $\Psi$ where each vertex
group of $\Psi$ is a subgroup of a conjugate of a vertex group of
$\Lambda$.
\end{corollary}

\begin{proof}
Applying theorem \ref{MT1} to $\Lambda$, we get a reduced
visual graph of groups $\Psi$ from $\Lambda$.  If $\Psi$ looks
irreducible with respect to $M(W)$ splittings, then we are done.
Otherwise, some vertex group of $\Psi$ visually splits nontrivially and
compatibly over an $M(W)$ special subgroup and we replace the
vertex with this visual splitting in $\Psi$. We can repeat,
replacing some special vertex group by special vertex groups with
fewer generators, until we must reach a visual graph of groups
which looks irreducible with respect to $M(W)$ splittings.
\end{proof}

Theorem \ref{Close} describes how ``close'' a
decomposition with $M(W)$ edge groups, which is irreducible with
respect to $M(W)$ splittings, is to a visual one.

\medskip

\noindent {\bf Theorem 2} {\it Suppose $(W,S)$ is a finitely generated Coxeter system and $\Lambda$ is a reduced
graph of groups decomposition of $W$ with $M(W)$ edge groups. If
$\Lambda$ is irreducible with respect to $M(W)$ splittings, and
$\Psi$ is a reduced graph of groups decomposition such that each edge group of $\Psi$ is in $M(S)$, each vertex group of $\Psi$ is a subgroup of a conjugate of a vertex group of  
$\Lambda$, and each edge group of $\Lambda$ contains a conjugate of an edge group of $\Psi$ (in particular if $\Psi$ is a reduced visual graph of groups decomposition for $(W,S)$ derived from $\Lambda$ as in the main theorem of \cite{MTVisual}), then
\begin{enumerate}
\item  $\Psi$ is irreducible with respect to $M(W)$ splittings

\item  There is a (unique) bijection $\alpha$ of the vertices
of $\Lambda$ to the vertices of $\Psi$ such that for each vertex
$V$ of $\Lambda$, $\Lambda(V)$ is conjugate to $\Psi(\alpha (V))$

\item  When $\Psi$ is visual, each edge group of $\Lambda$ is conjugate to a visual
subgroup for $(W,S)$.
\end{enumerate}}

\begin{proof}
Consider a vertex $V$ of $\Lambda$ with vertex group
$A=\Lambda(V)$. By theorem \ref{T1N}, $\Lambda (V)$ has a graph of groups decomposition $\Phi_V$ such that $\Phi _V$ is compatible with $\Lambda$, each edge group of $\Phi_V $ is in $M(W)$ and each vertex group of $\Phi_V$ is conjugate to a vertex group of $\Psi$ or conjugate to $\Lambda(E)$ for some edge $E$ of $\Lambda$ adjacent to $V$. Since $\Lambda$ is reduced and irreducible with respect to $M(W)$ splittings, $\Phi_V$ has a single vertex and $\Lambda(V)$ is conjugate to $\Psi(V')$ for some vertex $V'$ of $\Psi$. 

Since no vertex group of $\Psi$ is
contained in a conjugate of another, $V'$ is uniquely determined,
and we set $\alpha(V)=V'$. No vertex group of $\Lambda$ is conjugate to another so $\alpha$ is injective.  Since each vertex group $\Psi(V')$ is
contained in a conjugate of some $\Lambda(V)$ which is in turn
conjugate to $\Psi(\alpha(V))$ we must have $V'=\alpha(V)$ and
each $V'$ is in the image of $\alpha$.

If $\Psi$ is not irreducible with respect to $M(W)$ splittings,
then it does not look irreducible with respect to $M(W)$
splittings and some vertex group $W_1$ of $\Psi$ visually splits
nontrivially and compatibly over an $M(W)$ special subgroup $U_1$. 
Reducing gives a visual graph of groups decomposition $\Psi_1$ of $W$ satisfying the hypotheses on $\Psi$ in the statement of the theorem. Now $W_1$ is conjugate to a vertex group $A$
of $\Lambda$ and the above argument shows $A$ is conjugate to a vertex group of $\Psi_1$. But then, $W_1$ is conjugate to a vertex group of $\Psi_1$, which is nonsense.   Instead, $\Psi$ is irreducible with respect to $M(W)$
splittings.

Since $\Lambda$ is a tree, we can take each edge group of
$\Lambda$ as contained in its endpoint vertex groups taken as
subgroups of $W$. Hence each edge group is simply the
intersection of its adjacent vertex groups (up to conjugation).
Since vertex groups of $\Lambda$ are conjugates of
vertex groups in $\Psi$, their intersection is conjugate to a
special subgroup (by lemma \ref{Kil}) when $\Psi$ is visual.
\end{proof}

\begin{example}\label{simex2}
Let $W$ have the Coxeter presentation:
$$\langle s_1,s_2,s_3,s_4,s_5: s_k^2, (s_1s_2)^2, (s_2s_3)^2,
(s_3s_4)^2, (s_4s_5)^2\rangle \times \langle s_6,s_7:s_k^2\rangle
$$

\noindent Then $W$ is 1-ended and has the following visual
$M(W)$-irreducible decomposition (each edge group is 2-ended):
$$\!\langle s_1,s_2,s_6,s_7\rangle \ast
_{\langle s_2,s_6,s_7\rangle}\langle s_2,s_3,s_6,s_7\rangle\ast
_{\langle s_3,s_6, s_7\rangle }\langle s_3,s_4,s_6,s_7\rangle\ast
_{\langle s_4,s_6,s_7\rangle}\langle s_4, s_5, s_6,s_7\rangle\!$$
There is an automorphism of $W$ sending $s_5$ to $s_3s_5s_3$ and
all other $s_i $ to themselves. This gives another
$M(W)$-irreducible decomposition of $W$ where the last vertex
group $\langle s_4, s_5, s_6,s_7\rangle$ of the above graph of
groups decomposition is replaced by $\langle s_4, s_3s_5s_3,
s_6,s_7\rangle$. As $s_3$ does not commute with $s_1$ we see that
in regard to part 2 of theorem \ref{Close}, a single element of
$W$ cannot be expected to conjugate each vertex group of an
arbitrary $M(W)$-irreducible decomposition to a corresponding
vertex group of a corresponding visual $M(W)$-irreducible
decomposition.
\end{example}

\noindent {\bf Theorem 1}  {\it Finitely generated Coxeter groups
are strongly accessible over minimal splittings.}

\begin{proof}
Suppose $(W,S)$ is a finitely generated Coxeter system. There are
only finitely many elements of $K(W,S)$ (which includes the
trivial group).  For $G$ a subgroup of $W$ let $n(G)$ be the
number of elements of $K(W,S)$ which are contained in any
conjugate of $G$ (so $1\leq n(G)\leq n(W)$). For $\Lambda$ a
finite graph of groups decomposition of $W$, let
$c(\Lambda)=\Sigma_{i=1}^{n(W)} 3^ic_i(\Lambda)$ where $c_i(\Lambda)$ is the count of
vertex groups $G$ of $\Lambda$ with $n(G)=i$. 

 If $\Lambda$ reduces to $\Lambda'$ then
clearly $c_i(\Lambda ')\leq c_i(\Lambda)$ for all $i$, and for some $i$, $c_i(\Lambda ')$ is strictly less than $c_i(\Lambda)$. Hence, $c(\Lambda ')<c(\Lambda)$.  

If $\Lambda$ is reduced with $M(W)$ edge groups, and a
vertex group $G$ of $\Lambda$ splits non-trivially and compatibly as $A*_CB$ to
produce the decomposition $\Lambda'$ of $W$, then every subgroup of a conjugate of
$A$ or $B$ is a subgroup of a conjugate of $G$, but, by proposition \ref{Kreduce}, some element of $K(W,S)$ is contained in a conjugate of
$B$, and so of $G$, but not in a conjugate of $A$. Hence
$n(A)<n(G)$, and similarly $n(B)<n(G)$. This implies that
$c(\Lambda')<c(\Lambda)$ since $c_{n(G)}$ decreases by 1 in going from $\Lambda$ to $\Lambda'$  and the only other $c_i$ that change are $c_{n(A)}$ and $c_{n(B)}$, which are both increased by 1 if $n(A)\ne n(B)$ and $c_{n(A)}$ increases by 2 if $n(A)=n(B)$, but $c_{n(A)}$ and $c_{n(B)}$ have smaller coefficients than $c_{n(G)}$ in the summation $c$. More specifically, $c(\Lambda)-c(\Lambda')=3^{n(G)}-(3^{n(A)}+3^{n(B)})>0$. 

If $\Lambda$  is the trivial decomposition of $W$, then $c(\Lambda) =3^{\vert K(W,S)\vert}$ and we define this number to be $C(W,S)$.
Suppose $\Lambda_1,\ldots ,\Lambda_k$ is a sequence of reduced graph of groups decompositions of $W$ with $M(W)$ edge groups, such that $\Lambda_1$ is the trivial decomposition and $\Lambda_i$ is obtained from $\Lambda_{i-1}$ by splitting a vertex group $G$ of $\Lambda_{i-1}$ non-trivially and compatibly as $A\ast _CB$, for $C\in M(W)$ and then reducing. We have shown that $c(\Lambda_i)<c(\Lambda_{i-1})$ for all $i$, and so $k\leq C(W,S)$. In particular, 
$W$ is strongly accessible over $M(W)$ splittings
\end{proof}

\section{Generalizations, Ascending HNN extensions (and a group of Thompson) and  Closing questions}

Recall that if $G$ is a group and $H$ and $K$ are subgroups of $G$
then $H$ is smaller than $K$ if $H\cap K$ has finite index in $H$
and infinite index in $K$. Suppose $W$ is a finitely generated
Coxeter group and $\cal C$ is a class of subgroups of $W$ such
that for each $G\in \cal C$, any subgroup of $G$ is in $\cal C$,
e.g. the virtually abelian subgroups of $W$. Define $M(W,{\cal
C})$, {\it the minimal $\cal C$ splitting subgroups of $W$}, to
be the set of all subgroups $H$ of $W$ such that $H\in \cal C$,
$W$ splits non-trivially over $H$ and for any $K\in \cal C$ such
that $W$ splits non-trivially over $K$, $K$ is not smaller than
$H$. Then the same line of argument as used in this paper shows
that $W$ is strongly accessible over $M(W,{\cal C})$ splittings.

If $(W,S)$ is a finitely generated Coxeter system and $\Psi$ is
an $M(W)$-irreducible graph of groups decomposition of $W$ with
$M(W)$-edge groups, then by theorem \ref{Close}, each vertex group
$V$ of $\Psi$ is a Coxeter group with Coxeter system $(V,A)$
where $A$ is conjugate to a proper subset of $S$. The collection
$M(V)$ is not, in general, a subset of $M(W)$, and so $V$ has an
$M(V)$-irreducible graph of groups decomposition with $M(V)$-edge
groups. As $\vert A\vert <\vert S\vert$, there cannot be a
sequence $\Psi=\Psi_0,\Psi_1,\ldots ,\Psi_n$, with $n>\vert
S\vert$, of distinct graph of groups decompositions where $\Psi$
is $M(W)$-indecomposable with $M(W)$-edge groups, for $i>0$,
$V_i$ a vertex group of $\Psi_{i-1}$ and $\Psi_i$ is
$M(V_i)$-indecomposable with edge groups in $M(V_i)$. Such a
sequence must terminate with a special subgroup of $W$ that has
no non-trivial decomposition. By the FA results of
\cite{MTVisual}, that group must have a complete presentation
diagram.

Suppose $B$ is a group, and $\phi:A_1\to A_2$ is an isomorphism of subgroups of $B$. The group $G$ with presentation $\langle t,B:t^{-1}at=\phi(a) \hbox{ for } a\in A_1\rangle$ is called an {\it HNN extension} with {\it base group} $B$, {\it associated subgroups} $A_i$ and {\it stable letter} $t$. If $A_1=B$ then the HNN extension is  {\it ascending} and if additionally, $A_2$ is a proper subgroup of $B$ (i.e. $A_2\ne B$), then the HNN extension is {\it strictly ascending}. 

The bulk of this section is motivated by an example of Richard Thompson. Thompson's group $F$ is finitely presented and is an ascending HNN extension of a group isomorphic to $F$. Hence $F$ is not ``hierarchical accessible" over such splittings (see question 1 below).  If a group $G$ splits as an ascending HNN extension, then (by definition) there is no splitting of the base group which is compatible with the first splitting, so standard accessibility is not an issue. The only question is that of minimality of such splittings. 

\begin{theorem} \label{ind}
Suppose $A$ is a  finitely generated group and $\phi:A\to A$ is a monomorphism. Let $G\equiv \langle t, A:t^{-1}at=\phi (a) \hbox{ for }a\in A\rangle$ be the resulting ascending HNN extension. Then:

1)  If $\phi (A)$ has infinite index in $A$, this splitting of $G$ is not minimal and there is no finitely generated subgroup $B$ of $G$ such that $B$ is smaller than $A$, $G$ splits as an ascending HNN extension over $B$ and this splitting over $B$ is minimal.  

2) If $\phi(A)$ has finite index in $A$, then there is no finitely generated subgroup $B$ of $G$ such that $B$ is smaller than $A$ and $G$ splits as an ascending HNN extension over $B$.\end{theorem}
\begin{proof} 
First note that $G$ is also an ascending HNN extension over $\phi (A)$, (with presentation $\langle t, \phi (A): t^{-1}at=\phi (a) \hbox{ for } a\in \phi (A)\rangle$. Hence if $\phi(A)$ has infinite index in $A$, the splitting over $A$ is not minimal. Part 5) of  lemma \ref{asc} implies the second assertion of part 1) of the theorem. Part 4) of lemma \ref{asc} implies part 2) of the theorem. 
\end{proof}

\begin{lemma} \label{asc}
Suppose $\phi:A\to A$ and $\tau:B\to B$ are  monomorphisms of  finitely generated subgroups of $G$, and the corresponding ascending HNN extensions are isomorphic to $G$.
$$G\equiv \langle A,t:t^{-1}at=\phi(a) \hbox{ for all }a\in A\rangle\equiv \langle s,B:s^{-1}bs=\tau(b) \hbox { for all } b\in B\rangle$$ 
If $A\cap B$ has finite index in $B$ (so $B$ is potentially smaller than $A$). Then:

1)  The normal closures $N(A)$ and $N(B)$ in $G$ are equal. 

2) If $\phi(A)\ne A$ then $s=at$ for some $a\in N(A)$

3) If $\phi(A)=A$, then $\tau(B)=B$ (so $N(A)= A=B=N(B)$) and 
$s=at^{\pm 1}$ for some $a\in A$.

4) If $\phi(A)$ has finite index in $A$, then $A\cap B$ has finite index in $A$ (so $B$
is not smaller than $A$) and $\tau(B)$ has finite index in $B$.

5) If  $\phi(A)$ has infinite index in $A$, then $\tau(B)$ has infinite index in $B$. 
\end{lemma}
\begin{proof}
Let $A_0=A$ and let $A_i=t^iA_0t_i^{-i}$.  Then $t^{-1}A_it=A_{i-1}<A_i$.  Note that $N(A)=\cup _{i=0}^\infty A_i$.  Let $\pi:G\to G/N(A)\equiv \mathbb Z$ be the quotient map. Since $A\cap B$ has finite index in $B$, $\pi(B)$ is finite (and hence trivial). This implies $B<N(A)$. As $B$ is finitely generated, $B<A_m$ for some $m$. This also implies that $\langle \pi(s)\rangle=\langle \pi(t)\rangle=\mathbb Z$ and so $N(B)=N(A)$, completing 1).  

Normal forms in ascending HNN extensions imply $s=t^pa_1t^{-q}$ for some $p,q\geq 0$ and $a_1\in A$. This implies $|p-q|=1$. Hence $s=at^{\pm 1}$ for $a\in A_p$. 

Suppose $s=at^{-1}$. Let $r$ be the maximum of $m$ and $p$. Note that 
$$N(A)=N(B)=\cup_{i=0}^{\infty}s^iBs^{-i}=\cup _{i=0}^\infty (at^{-1})^iB(ta^{-1})^i<A_r$$ 
(since, $t^{-1}Bt<t^{-1}A_mt=A_{m-1}<A_r$ and (as $a\in A_r$) $aA_ra^{-1}=A_r$). But if $\phi(A)\ne A$, $A_{r+1}\not <  A_r$. Instead, $s=at$, completing 2).

If $\phi(A)=A$, then $N(A)=A$. As $N(B)=A$ is finitely generated, $N(B)=\cup _{i=0}^n s^iBs^{-i}=s^nBs^{-n}$ for some $n>0$. So $N(B)=B$ completing 3).

If $\phi(A)$ has finite index in $A$ then $A$ has finite index in $A_i$ for all $i\geq 0$. Since $N(A)=N(B)$, and $A$ and $B$ are finitely generated, there are positive integers $p<p'$ and $q<q'$ such that
$$A<s^pBs^{-p}<t^qAt^{-q}<s^{p'}Bs^{-p'}<t^{q'}At^{-q'}$$
Hence $B$ has finite index in $s^{p'-p}Bs^{-(p'-p)}$. This implies $B$ has finite index in $s^iBs^{-i}$ for all $i\geq 0$ and also $\tau (B)$ has finite index in $B$. Similarly, there are positive integers $j$ and $k$ such that
$$B<t^kAt^{-k}<s^jBs^{-j}$$
Hence $B$ and $A$ (and so $A\cap B$) have finite index in $t^kAt^{-k}$. This implies $A\cap B$ has finite index in $A$ and 4) is complete.

Assume $\phi(A)$ has infinite index in $A$. As $A$ and $B$ are finitely generated subgroups of $N(A)=N(B)$, there positive integers $k$ and $j$ such that 
$$B<t^kAt^{-k}<s^jBs^{-j}$$
The group $B$ does not have finite index in $t^kAt^{-k}$ since otherwise $A\cap B$ (and then $A$) would have finite index in $t^kAt^{-k}$. This implies $B$ has infinite index in $s^jBs^{-j}$. This in turn implies $B$ has infinite index in $s^iBs^{-i}$ for all $i\geq 0$. This also implies $\tau(B)$ has infinite index in $B$. 
\end{proof}

\begin{example} {\bf (Thompson's Group)}
In unpublished work, R. J. Thompson introduced a group, traditionally denoted $F$, in the context of finding infinite finitely presented simple groups. This group is now well studied in a variety of other contexts. The group $F$ has presentation 
$$\langle x_1,x_2,\ldots : x_i^{-1}x_jx_i=x_{j+1} \hbox { for } i<j\rangle$$ 
Well know facts about this group include: $F$  is $FP_\infty$ (\cite {BG}), in particular, $F$ is finitely presented (with generators $x_1$ and $x_2$), the commutator subgroup of $F$ is simple (\cite{Br}), and $F$ contains no free group of rank 2 (\cite{BS}). Clearly, $F$ is an ascending HNN extension of itself (with base group $\langle x_2,x_3,\ldots \rangle$ and stable letter $x_1$ - called the ``standard" splitting of $F$).

We are interested in understanding ``minimal" splittings of $F$ and more generally minimal splittings of finitely generated groups containing no non-abelian free group. We list some elementary facts.

\medskip

\noindent {\bf Fact 1.} {\it If $G$ contains no non-abelian free group and $G$ splits as an amalgamated free product $A\ast_CB$ then $C$ is of index 2 in both $A$ and $B$ and hence is normal in $G$. If $G$ splits as an HNN-extension, then this splitting is ascending. }

\medskip



\medskip

\noindent{\bf Fact 2.}
{\it The group $F$ does not split non-trivially as $A\ast_CB$}

\begin{proof}
Otherwise $C$ is normal in $F$ and $F/C$ is isomorphic to $\mathbb Z_2\ast \mathbb Z_2$. Since the commutator subgroup $K$ of $G$ is simple, $K\cap C$ is either trivial or $K$. The intersection is not $K$ since $F/C$ is not abelian. The intersection is non-trivial, since otherwise $K$ would inject under the quotient map $F\to F/C\equiv \mathbb Z_2\ast \mathbb Z_2$.
\end{proof}




By theorem \ref{ind} and the previous facts we have:

\medskip

\noindent{\bf Fact 3.}
{\it The only non-trivial splittings of $F$ are as ascending HNN extensions $\langle t,A:t^{-1}at=\phi(a) \hbox{ for all } a\in A\rangle$. For $A$ finitely generated, this splitting is minimal iff the image of the monomorphism $\phi:A\to A$ has finite index in $A$.}

\medskip



R. Bieri, W. D. Neumann and R. Strebel have shown that if $G$ is a finitely presented group containing no free group of rank 2 and $G$ maps onto $\mathbb Z\oplus \mathbb Z$, then $G$ contains a finitely generated normal subgroup $H$ such that $G/H\cong \mathbb Z$ (see theorem D of \cite{BNS} or theorem 18.3.8 of \cite{Ge}). Hence, there is a short exact sequence $1\to H\to F\to\mathbb Z\to1$ with $H$ finitely generated. 

\medskip

\noindent {\bf Fact 4.} {\it The ascending HNN extensions given by the short exact sequence $1\to H\to F\to\mathbb Z\to1$ (with $H$ finitely generated) are minimal splittings.}



\begin{theorem} Suppose $G$ is a finitely generated group containing  no non-abelian free subgroup. Suppose $G$ can be written as an ascending HNN extension $\langle t, A:t^{-1}at=\phi(a) \hbox{ for all } a\in A\rangle$ and as non-trivial amalgamated products $C\ast_DE$ and $H\ast_KL$  where all component groups are finitely generated, then:

1) $D\cap A$ does not have finite index in $A$ or  $D$ (so neither $A$ nor $D$ is smaller than the other),

2) if $D\cap K$ has finite index in $K$ then $K=D$ (so neither $D$ nor $K$ is smaller than the other),  

3) $C\ast_DE$ is a minimal splitting and $\langle t, A:t^{-1}at=\phi(a) \hbox{ for all } a\in A\rangle$ is minimal iff $\phi(A)$ has finite index in $A$.
\end{theorem}
 
\begin{proof}
Let $q:G\to G/N(A)\equiv \mathbb Z$ and $p:G\to G/D\equiv \mathbb Z_2\ast \mathbb Z_2$ be the quotient maps.
If $D\cap A$ has finite index in $D$ then $q(D)$ is finite, so $q(D)$ is trivial and $D<N(A)$. But this implies there is a homomorphism from $\mathbb Z_2\ast \mathbb Z_2$ onto $\mathbb Z$, which is nonsense. 

If $D\cap A$ has finite index in $A$, then $p(A)$ is a finite subgroup of $\mathbb Z_2\ast \mathbb Z_2\equiv \langle x:x^2=1\rangle \ast \langle y:y^2=1\rangle$. Then $p(A)$ is a subgroup of a conjugate of $\langle x\rangle$ or $ \langle y\rangle$. Without loss, assume $p(A)<\langle x\rangle$. If $p(A)=1$, then $A<D$ and so $N(A)<D$. But this implies there is a homomorphism of $\mathbb Z$ onto $\mathbb Z_2\ast \mathbb Z_2$ which is nonsense. Hence $p(A)=\langle x\rangle$. But then $p(t)$ commute with $x$. This is implies $p(t)$ is trivial. This is impossible as $p(t)$ and $p(A)$ generate $\mathbb Z_2\ast \mathbb Z_2$. Part 1) is finished.

Suppose $D\cap K$ has finite index in $K$. Then as above we can assume that $p(K) <\langle x\rangle$. If $p(K)=1$, then $K<D$. If additionally $K\ne D$ then there is a homomorphism from $\mathbb Z_2\ast \mathbb Z_2$ onto $\mathbb Z_2\ast \mathbb Z_2$ with non-trivial kernel. This is impossible. Hence, either, $K=D$ or $p(K)=\langle x\rangle$. We conclude that $K=D$ since $p(K)$ is normal in $\mathbb Z_2\ast \mathbb Z_2$, and 2) is finished.

Fact 1, and part 2) implies $C\ast_DE$ is a minimal splitting. Fact 1, theorem \ref{ind} and part 1) imply $\langle t, A:t^{-1}at=\phi(a) \hbox{ for all } a\in A\rangle$ is minimal iff $\phi(A)$ has finite index in $A$. 
\end{proof}

 \end{example}

We conclude this paper with some questions of interest.

\begin{enumerate}
\item
For an arbitrary finitely generated Coxeter group $W$, is there a
sequence $\Lambda _1, \Lambda_2,\dots $ of graphs of groups such
that $\Lambda _1$ is a non-trivial decomposition of $W$ with edge
groups in $M(W)$, and for $i>1$, $\Lambda _i$ is a non-trivial
decomposition of a vertex group $V_i$ of $\Lambda _{i-1}$ with
$M(V_i)$-edge groups (but $\Lambda_i$ is not necessarily
compatible with $\Lambda _{i-1}$)? This sort of accessibility is called {\it  hierarchical} accessibility in analogy with 3-manifold decompositions). If no such sequence exists,
then does a last term of such a splitting sequence have no
splittings of any sort (is it FA)? Would such a last term always
be visual?

\item
Is there a JSJ theorem for Coxeter groups over minimal
splittings? In \cite{MTJSJ}, we produce a JSJ result for Coxeter
groups over virtually abelian splitting subgroups that relies on splittings over minimal virtually abelian subgroups.

\medskip

For the standard strictly ascending HNN splitting of Thompson's group $F$ (given by $\langle x_1, x_2,\ldots :x_i^{-1}x_jx_i=x_{j+1} \hbox{ for } i<j\rangle$ - with base group $B\equiv\langle x_2,x_3,\ldots \rangle$ and stable letter $x_1$) there is no minimal splitting subgroup $C$ of $F$
with $C$ smaller than $B$. Hence, for finitely presented groups,  there is no analogue for proposition \ref{L24N}. Still $F$, and in fact all finitely generated groups containing no non-abelian free group, are strongly accessible over finitely generated  minimal splitting subgroups. 

\item
Are finitely presented groups (strongly) accessible over finitely generated 
minimal splittings?
 
\medskip

Finitely generated groups are not accessible over finite splitting subgroups (see D2), and hence finitely generated groups are not accessible over minimal splittings.
 
\item Does Thompson's group split as a {\it strict} ascending HNN  extension with finitely generated base $A$ and monomorphism $\phi:A\to A$ such that $\phi(A)$ has finite index in $A$?
\end{enumerate}

\enddocument